\documentclass[10pt,oneside]{amsart}

\usepackage{amssymb, amsmath, amsthm, amsfonts}
\usepackage{mathrsfs,comment}
\usepackage{graphicx}
\usepackage{geometry}
\usepackage{float}
\usepackage[all]{xy}
\usepackage{xcolor}
\usepackage{tikz}

\usepackage[bookmarks,bookmarksopen,bookmarksdepth=4]{hyperref}
\usepackage{url}
\usepackage{enumerate}
\usepackage{color}

\usepackage{verbatim}
\usepackage[T1]{fontenc}

\hypersetup{%
  bookmarksnumbered=true,%
  colorlinks=true,%
  linkcolor=blue,%
  citecolor=blue,%
  filecolor=blue,%
  menucolor=blue,%
  urlcolor=blue,%
  pdfnewwindow=true,%
  pdfstartview=FitBH}

\def\@url#1{{\tt\def~{\lower3.5pt\hbox{\char'176}}\def\_{\char'137}#1}}

\newtheorem{thm}{Theorem}[section]
\newtheorem{cor}[thm]{Corollary}
\newtheorem{prop}[thm]{Proposition}
\newtheorem{lem}[thm]{Lemma}
\theoremstyle{definition}
\newtheorem{defn}[thm]{Definition}

\newtheorem{ex}[thm]{Example}

\newtheorem{example}[thm]{Example}

\newtheorem{rem}[thm]{Remark}
\newtheorem{rmk}[thm]{Remark}

\makeatletter
\let\c@lem=\c@thm
\let\c@cor=\c@thm
\let\c@prop=\c@thm
\let\c@lem=\c@thm
\let\c@defn=\c@thm
\let\c@exmps=\c@thm
\let\c@rem=\c@thm
\let\c@warn=\c@thm
\let\c@claim=\c@thm
\makeatother

\numberwithin{equation}{subsection}

\def\endash{\mathchar"2D}

\newcommand{\cell}{\endash \textnormal{cell} \endash}

\newcommand{\leftmod}{\endash \textnormal{mod}}

\newcommand{\A}{\mathsf{A}}
\newcommand{\RR}{\mathbb{R}}

\newcommand{\Mdef}[2]{\newcommand{#1}{\relax \ifmmode #2 \else $#2$\fi}}
\Mdef{\bA}{\mathbb{A}}
\Mdef{\bB}{\mathbb{B}}
\Mdef{\bC}{\mathbb{C}}
\Mdef{\bD}{\mathbb{D}}
\Mdef{\bE}{\mathbb{E}}
\Mdef{\bF}{\mathbb{F}}
\Mdef{\bG}{\mathbb{G}}
\Mdef{\bH}{\mathbb{H}}
\Mdef{\bI}{\mathbb{I}}
\Mdef{\bJ}{\mathbb{J}}
\Mdef{\bK}{\mathbb{K}}
\Mdef{\bL}{\mathbb{L}}
\Mdef{\bM}{\mathbb{M}}

\Mdef{\bO}{\mathbb{O}}
\Mdef{\bP}{\mathbb{P}}
\Mdef{\bQ}{\mathbb{Q}}
\Mdef{\bR}{\mathbb{R}}
\Mdef{\bS}{\mathbb{S}}
\Mdef{\bT}{\mathbb{T}}
\Mdef{\bU}{\mathbb{U}}
\Mdef{\bV}{\mathbb{V}}
\Mdef{\bW}{\mathbb{W}}
\Mdef{\bX}{\mathbb{X}}
\Mdef{\bY}{\mathbb{Y}}
\Mdef{\bZ}{\mathbb{Z}}

\Mdef{\mcA}{\mathcal{A}}
\Mdef{\mcB}{\mathcal{B}}
\Mdef{\mcC}{\mathcal{C}}
\Mdef{\mcD}{\mathcal{D}} 
\Mdef{\mcE}{\mathcal{E}}
\Mdef{\mcF}{\mathcal{F}}
\Mdef{\mcG}{\mathcal{G}}
\Mdef{\mcH}{\mathcal{H}} 
\Mdef{\mcI}{\mathcal{I}}
\Mdef{\mcJ}{\mathcal{J}}
\Mdef{\mcK}{\mathcal{K}}
\Mdef{\mcL}{\mathcal{L}}

\Mdef{\mcM}{\mathcal{M}}
\Mdef{\mcN}{\mathcal{N}}
\Mdef{\mcO}{\mathcal{O}}
\newcommand{\cO}{\mathcal O}
\newcommand{\cP}{\mathcal{P}}
\newcommand{\cC}{\mathcal C}
\newcommand{\cT}{\mathcal T}
\Mdef{\mcP}{\mathcal{P}}
\Mdef{\mcQ}{\mathcal{Q}}
\Mdef{\mcR}{\mathcal{R}}
\Mdef{\mcS}{\mathcal{S}}
\Mdef{\mcT}{\mathcal{T}}
\Mdef{\mcU}{\mathcal{U}}
\Mdef{\mcV}{\mathcal{V}}
\Mdef{\mcW}{\mathcal{W}}
\Mdef{\mcX}{\mathcal{X}}
\Mdef{\mcY}{\mathcal{Y}}
\Mdef{\mcZ}{\mathcal{Z}}

\Mdef{\At}{\tilde{A}}
\Mdef{\Bt}{\tilde{B}}
\Mdef{\Ct}{\tilde{C}}
\Mdef{\Et}{\tilde{E}}
\Mdef{\Ht}{\tilde{H}}
\Mdef{\Kt}{\tilde{K}}
\Mdef{\Lt}{\tilde{L}}
\Mdef{\Mt}{\tilde{M}}
\Mdef{\Nt}{\tilde{N}}
\Mdef{\Pt}{\tilde{P}}

\newcommand{\tcT}{\widetilde{\mathcal T}}
\newcommand{\tcD}{\widetilde{\mathcal D}}

\Mdef{\tA}{\tilde{A}}
\Mdef{\tB}{\tilde{B}}
\Mdef{\tC}{\tilde{C}}
\Mdef{\tE}{\tilde{E}}
\Mdef{\tH}{\tilde{H}}
\Mdef{\tK}{\tilde{K}}
\Mdef{\tL}{\tilde{L}}
\Mdef{\tM}{\tilde{M}}
\Mdef{\tN}{\tilde{N}}
\Mdef{\tP}{\tilde{P}}

\Mdef{\ft}{\tilde{f}}
\Mdef{\xt}{\tilde{x}}
\Mdef{\yt}{\tilde{y}}

\Mdef{\Ab}{\overline{A}}
\Mdef{\Bb}{\overline{B}}
\Mdef{\Cb}{\overline{C}}
\Mdef{\Db}{\overline{D}}
\Mdef{\Eb}{\overline{E}}
\Mdef{\Fb}{\overline{F}}
\Mdef{\Gb}{\overline{G}}
\Mdef{\Hb}{\overline{H}}
\Mdef{\Ib}{\overline{I}}
\Mdef{\Jb}{\overline{J}}
\Mdef{\Kb}{\overline{K}}
\Mdef{\Lb}{\overline{L}}
\Mdef{\Mb}{\overline{M}}
\Mdef{\Nb}{\overline{N}}
\Mdef{\Ob}{\overline{O}}
\Mdef{\Pb}{\overline{P}}
\Mdef{\Qb}{\overline{Q}}
\Mdef{\Rb}{\overline{R}}
\Mdef{\Sb}{\overline{S}}
\Mdef{\Tb}{\overline{T}}
\Mdef{\Ub}{\overline{U}}
\Mdef{\Vb}{\overline{V}}
\Mdef{\Wb}{\overline{W}}
\Mdef{\Xb}{\overline{X}}
\Mdef{\Yb}{\overline{Y}}
\Mdef{\Zb}{\overline{Z}}

\Mdef{\db}{\overline{d}}
\Mdef{\hb}{\overline{h}}
\Mdef{\qb}{\overline{q}}
\Mdef{\rb}{\overline{r}}
\Mdef{\tb}{\overline{t}}
\Mdef{\ub}{\overline{u}}
\Mdef{\vb}{\overline{v}}

\Mdef{\hc}{\hat{c}}
\Mdef{\he}{\hat{e}}
\Mdef{\hf}{\hat{f}}
\Mdef{\hA}{\hat{A}}
\Mdef{\hH}{\hat{H}}
\Mdef{\hJ}{\hat{J}}
\Mdef{\hM}{\hat{M}}
\Mdef{\hP}{\hat{P}}
\Mdef{\hQ}{\hat{Q}}

\Mdef{\thetab}{\overline{\theta}}
\Mdef{\phib}{\overline{\phi}}

\Mdef{\uA}{\underline{A}}
\Mdef{\uB}{\underline{B}}
\Mdef{\uC}{\underline{C}}
\Mdef{\uD}{\underline{D}}

\Mdef{\bolda}{\mathbf{a}}
\Mdef{\boldb}{\mathbf{b}}
\Mdef{\boldD}{\mathbf{D}}

\newcommand{\lra}{\longrightarrow}

\newcommand{\Ch}{\mathrm{Ch}}

\newcommand{\co}{\colon}

\Mdef{\av}{\mathrm{av}}
\Mdef{\infl}{\mathrm{inf}}
\Mdef{\defl}{\mathrm{def}}
\Mdef{\res}{\mathrm{res}}
\Mdef{\ind}{\mathrm{ind}}
\Mdef{\coind}{\mathrm{coind}}
\Mdef{\id}{\mathrm{Id}}

\newcommand{\adjunction}[4]{
\xymatrix{
#1:#2 \ar@<0.7ex>[r] &
\ar@<0.7ex>[l] #3:#4
}}

\newcommand{\reverseadjunction}[4]{
\xymatrix{
#1:#2 \ar@<-0.7ex>[r] &
\ar@<-0.7ex>[l] #3:#4
}}

\definecolor{orange}{RGB}{255, 127, 0}

\definecolor{darkgreen}{RGB}{35, 127, 89}

\newcommand{\T}{ {\mathbb{T}} }

\newcommand{\smashprod}{\wedge}

\newcommand{\Hom}{\mathrm{Hom}}

\Mdef{\bhom}{\mathbf{\hat{H}om}}

\Mdef{\Mod}{\mathrm{mod}}

\newcommand{\colimit}{\mathrm{colim}}

\newcommand{\mackey}[1]{\textnormal{Mackey}(#1)}
\newcommand{\barmackey}[1]{\overline{\textnormal{Mackey}}(#1)}
\newcommand{\mackeyfunctor}[1]{\textnormal{Mack}_{#1}}
\newcommand{\trivialgroup}{{\{1\}}}
\newcommand{\trivgp}{\trivialgroup}
\newcommand{\burnsidering}{\A_\bQ}

\newcommand{\preextend}{i_\#}
\newcommand{\preforget}{i^\#}
\newcommand{\otherfunctor}{\varepsilon_\#}
\newcommand{\preinflate}{\varepsilon^\#}

\title{An introduction to algebraic models
for rational \texorpdfstring{$G$}{G}-spectra}
\author[Barnes]{David Barnes}
\address[Barnes]{Mathematical Sciences Research Centre, Queen's University Belfast}
\author[K\k{e}dziorek]{Magdalena K\k{e}dziorek}
\address[K\k{e}dziorek]{IMAPP, Radboud University Nijmegen}

\begin{document}
\begin{abstract}
The project of Greenlees et al.\ on understanding rational $G$-spectra in terms of algebraic categories has had many successes, classifying rational $G$-spectra for finite groups, $SO(2)$, $O(2)$, $SO(3)$, free and cofree $G$-spectra as well as
rational  toral $G$-spectra for arbitrary compact Lie groups.

This paper provides an introduction to the subject in two parts.  
The first discusses rational $G$-Mackey functors, the action of the Burnside ring and change of group functors.
It gives a complete proof of the well-known classification of rational Mackey functors for finite $G$.  
The second part discusses the methods and tools from equivariant stable homotopy theory 
needed to obtain algebraic models for rational $G$-spectra.
It gives a summary of the key steps in the classification of rational $G$-spectra
in terms of a symmetric monoidal algebraic category. 

Having these two parts in the same place allows one to clearly see 
the analogy between the algebraic and topological classifications.
\end{abstract}

\maketitle

\setcounter{tocdepth}{1}
\tableofcontents

\section{Introduction}

Suppose $G$ is a compact Lie group. 
The project of  understanding the homotopy theory of rational $G$-spectra
in terms of algebraic categories was started by Greenlees in 1999 \cite{gre99}. It has had many successes since, classifying
rational $G$-spectra for finite groups, $SO(2)$, $O(2)$, $SO(3)$,
free and cofree $G$-spectra as well as rational toral $G$-spectra for an arbitrary compact Lie group $G$.
The project has expanded to consider (commutative) ring spectra in terms of these algebraic models.
This paper provides an introduction to this body of work,
whose papers often assume a deep familiarity with rational equivariant homotopy theory.

Starting from the definition of rational $G$-Mackey functors for a finite group $G$, we explain how
the rational Burnside ring acts on this category and how change of groups functors
behave. Combining these functors, we give an accessible account of the structure and
classification of rational $G$-Mackey functors in terms of group rings
and a comparison of the monoidal structures. We explain how this classification is the
template for the classifications of rational $G$-spectra for varying $G$.

The second half of the paper considers rational $G$-spectra for $G$ a compact Lie group.
Here the rational Burnside ring appears as the ring of self maps
of the sphere spectrum. We describe the structure of this ring and
its idempotents. Following the template, we show how the same approach
(Burnside ring actions, restriction to subgroups and fixed points)
is used in the various classifications of rational $G$-spectra.
We also discuss the additional complexities (isotropy separation,
localisations and cellularisations) that are needed for spectra.

The conjecture by Greenlees states that for any compact Lie group $G$ there is a nice graded abelian category $\mcA(G)$, such that the category $d\mcA(G)$ of differential objects in $\mcA(G)$\footnote{In other words objects of $\mcA(G)$ equipped with a differential} with a certain model structure is Quillen equivalent to the category of rational $G$-spectra
$$G\endash\text{Sp}_\bQ \simeq_Q d\mcA(G).$$
\emph{Nice} here means that the category $\mcA (G)$ is of homological dimension (that is, injective dimension) equal to the rank of $G$ and of a form that is easy to use in calculations. 
If we find such $\mcA(G)$ and  $d\mcA(G)$ equipped with a model structure Quillen equivalent to $G\endash\text{Sp}_\bQ$, we say that $\mcA(G)$ is an \emph{abelian model} and $d\mcA(G)$ is an \emph{algebraic model} for rational $G$-spectra. The conjecture
is known for quite a number of groups in some form.
Particularly useful examples are the case of $O(2)$ as given in  \cite{barneso2} and \cite{greo2}; and 
$SO(3)$ as given in \cite{KedziorekSO(3)} and \cite{GreenleesSO3}. 
We refer to the introduction of \cite{tnqcore} for a more complete summary of the known cases. 

Since \cite{tnqcore} was published, there have been significant developments in the field.
This includes extending the existence of algebraic models to profinite groups 
(see \cite{barnessugrue20}
and \cite{sugruethesis}) as well as taking various complexities with monoidal structure 
into account (see \cite{BGKeinfty}, \cite{BGKso2einfty} and \cite{JordanLucaCofree}).
We refer the reader to \cite{BalchinGreenlees} for a related result stating that a nice stable, monoidal model category has a model built from categories of modules over completed rings in an adelic fashion.
Recently, monoidal algebraic models (when $G$ is a finite abelian group) were used in establishing the uniqueness of naive and genuine commutative ring structures on the rational equivariant complex $K$-theory spectrum, 
see Bohmann et al.\  \cite{WIT3}, \cite{WIT3_paper2}.

The aim of this paper is to give a new introduction and explanation to some of these existing results
while demonstrating the analogy between the algebraic and topological sides. 
By doing so, we intend to give an 
overview of the methods and tools used in obtaining algebraic models for rational $G$-spectra and provide a step-by-step guide, at least in some cases. 

\subsection*{Acknowledgements}
The second author is grateful for support from the Dutch Research Council (NWO)
under Veni grant 639.031.757.

\part{The structure of rational Mackey functors}\label{part:1}

\section{An introduction to rational Mackey functors}

For $G$ a finite group, the category of Mackey functors is an abelian category
that is important to group theorists 
(see Nakaoka \cite{Nakaoka16} for example)
and algebraic topologists working equivariantly.
Working over the rationals greatly simplifies the category,
rationally it splits into
a direct product of modules over group rings of the
Weyl groups of subgroups of $G$ (counted up to conjugacy).
We use the rationals for definiteness, but the splitting holds when working 
over any commutative ring $R$ where $|G|^{-1} \in R$.

This result is stated formally as Theorem \ref{thm:mainsplitting}.
It was proven independently by two sources,
Greenlees and May \cite[Appendix A]{gremay95}
and Th\'{e}venaz and Webb \cite[Theorems 8.3 and 9.1]{tw90} 
The former took an approach from equivariant stable homotopy theory, the latter
from algebra.
We find the former approach simpler, so we follow it,
expanding substantially on the proofs.
General references for the results on Mackey functors are
Greenlees \cite{projremarks92},
Greenlees and May \cite{GM92structure}, 
Th\'{e}venaz and Webb \cite{tw95}
and
Webb \cite{webb00}.
For a discussion on Mackey functors for compact Lie groups see \cite{LewisMackeyConpactG}.

From the many equivalent definitions of a Mackey functor, we choose
one in terms of induction and restriction maps.
\begin{defn}
A \emph{rational $G$-Mackey functor}\index{Mackey functor} $M$ is:
\begin{itemize}
\item a collection of $\mathbb{Q}$-modules $M(G/H)$ for each subgroup $H \leqslant G$,
\item for subgroups $K,H \leqslant G$ with $K\leqslant H$ and any $g \in G$
we have a \emph{restriction} map, an \emph{induction} map and a \emph{conjugation} map
\[
R^H_K \co M(G/H)\rightarrow M(G/K), \quad
I^H_K \co M(G/K)\rightarrow M(G/H) \quad \text{ and } \quad
C_g \co M(G/H)\rightarrow M(G/gHg^{-1}).
\]
\end{itemize}
These maps satisfy the following conditions.

\begin{enumerate}
\item For all subgroups $H$ of $G$ and all  $h\in H$
\[
R^H_H=\id_{M(G/H)}=I^H_H \quad \text{ and } \quad  C_h=\id_{M(G/H)}.
\]

\item For $L\leqslant K\leqslant H$ subgroups of $G$ and $g,h\in G$, there are composition rules
\[
I^H_L=I^H_K\circ I^K_L, \qquad R^H_L=R^K_L\circ R^H_K, \quad \text{ and } \quad  C_{gh}=C_g\circ C_h.
\]
The first two are \emph{transitivity} of induction and restriction. The last is \emph{associativity} of conjugation.

\item For $g\in G$ and $K\leqslant H$ subgroups of $G$, there are composition rules
\[
R^{gHg^{-1}}_{gKg^{-1}}\circ C_g=C_g\circ R^H_K  \quad \text{ and } \quad  I^{gHg^{-1}}_{gKg^{-1}}\circ C_g=C_g\circ I^H_K.
\]
This is the \emph{equivariance} of restriction and induction.

\item For subgroups $K,L\leqslant H$ of $G$
\[
R^H_K\circ I^H_L=\sum_{x\in[K\diagdown H \diagup L]} I^K_{K\cap xLx^{-1}}\circ C_x\circ R^L_{L\cap x^{-1}Kx}.
\]
This condition is known as the \emph{Mackey axiom}.
\end{enumerate}

We denote the category of rational Mackey functors by $\mackey{G}$.
\end{defn}

To save space, many texts shorten the input and write $M(H) := M(G/H)$.
This notation aligns better with the terms induction and restriction, but precludes the following
remark.

\begin{rmk}
Since every finite $G$-set is (up to non-canonical isomorphism) a disjoint union of orbits $G/H$,
we can (by choosing such an isomorphism) extend any Mackey functor to take input
from the category of finite $G$-sets and $G$-maps by sending disjoint union to
direct sums. We will repeatedly use this extension (without further notice) in the adjunctions on Mackey functors that we define later.
\end{rmk}

Lindner \cite{lindner76} uses this extension to give an equivalent definition
of Mackey functors in terms of
a pair of covariant and contravariant functors from finite $G$-sets to
$\mathbb{Q}$-modules. These functors agree on objects,
send disjoint unions to direct sums
and satisfy a pullback condition (that is equivalent to the
Mackey axiom). The equivalence is proven via the decomposition
\[
G/K \times G/H = \coprod_{x\in[K\diagdown H \diagup L]} G/ (H \cap xKx^{-1}) .
\]
A further definition in terms of spans of $G$-sets (the Burnside category) is also given in that reference.

We illustrate how the structure works for two small groups.
\begin{ex}
Let $G=C_2 = \{1, \sigma \}$. A rational Mackey functor is a pair of $\bQ$-modules
$M(C_2/C_2)$ and $M(C_2/\trivgp)$. The conjugation maps imply that both $\bQ$-modules
have an action of $C_2$, but it is trivial on the first module.
There is a restriction map, which commutes with the $C_2$-actions
\[
M(C_2/C_2) \lra M(C_2/1)^{C_2} \hookrightarrow M(C_2/\trivgp).
\]
Similarly there is an induction map, which commutes with the $C_2$-actions
\[
M(C_2/\trivgp) \lra M(C_2/\trivgp)/{C_2} \lra M(C_2/C_2).
\]
The Mackey axiom (for $H=C_2$, $K=L= \trivgp$) says that
\[
R_\trivgp^{C_2} \circ  I_\trivgp^{C_2}
= \sum_{x\in[\trivgp \diagdown C_2 \diagup \trivgp]} I^\trivgp_{\trivgp}\circ C_\trivgp \circ R^\trivgp_{\trivgp}
= \sum_{x\in  C_2 } C_x = \id + C_{\sigma}
\]
\end{ex}

\begin{ex}
Let $G=C_6$. A rational $C_6$-Mackey functor consists of  four $\bQ$-modules with maps between them.
We draw this as a Lewis diagram below.
The looped arrows indicate the group that acts on each module.
Section \ref{sec:c6examples} gives several examples of rational $C_6$-Mackey functors.  

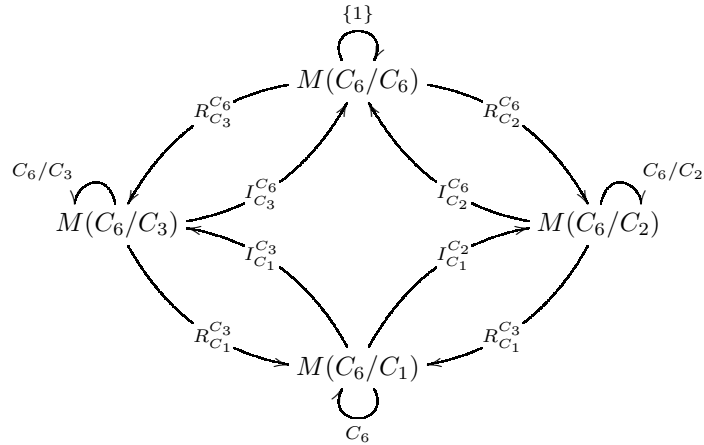
\[
\xymatrix@R+0.5cm@C+0.5cm{
& M(C_6/C_6)
\ar@(ul,ur)^{\trivgp}
\ar@/^1.5pc/[dr]|{R_{C_2}^{C_6}}
\ar@/_1.5pc/[dl]|{R_{C_3}^{C_6}}
\\
M(C_6/C_3)
\ar@(u,l)_{C_6/C_3}
\ar@/_1.5pc/[dr]|{R_{C_1}^{C_3}}
\ar@/_1.5pc/[ur]|{I_{C_3}^{C_6}}
&&
M(C_6/C_2)
\ar@(u,r)^{C_6/C_2}
\ar@/^1.5pc/[dl]|{R_{C_1}^{C_3}}
\ar@/^1.5pc/[ul]|{I_{C_2}^{C_6}}
\\
&
M(C_6/C_1)
\ar@(dr,dl)^{C_6}
\ar@/^1.5pc/[ur]|{I_{C_1}^{C_2}}
\ar@/_1.5pc/[ul]|{I_{C_1}^{C_3}}
}
\]
The Mackey axiom also implies that
\[
R_{C_3}^{C_6} \circ I_{C_2}^{C_6} =
I_{C_1}^{C_3} \circ R_{C_1}^{C_2}
\quad
\textrm{ and }
\quad
R_{C_2}^{C_6} \circ I_{C_3}^{C_6} =
I_{C_1}^{C_2} \circ R_{C_1}^{C_3}.
\]
\end{ex}

There are several general constructions that give examples of Mackey functors.
\begin{ex}\label{ex:constantMackey}
The \emph{constant Mackey functor} at a $\bQ$-module $A$ takes value $A$ at
each $G/H$. The conjugation and restriction maps are the identity map of $A$, induction from $G/K$ to $G/H$ is multiplication by the index of $K$ inside $H$. Given that the restriction maps are identities, the Mackey axiom prevents the
induction maps from being identity maps.

We may also define the
\emph{co-constant Mackey functor} at a $\bQ$-module $A$ takes value $A$ at
each $G/H$. The conjugation and induction maps are the identity of $A$
and restriction from $G/H$ to $G/K$
is multiplication by the index of $K$ inside $H$.
\end{ex}

The similarity between the constant and co-constant Mackey functors is an example of duality of Mackey functors. 
That is, the co-constant Mackey functor is dual to the constant one, in the following sense.
\begin{lem}
Given a Mackey functor $M$, there is a \emph{dual Mackey functor} $DM$,
that at $G/H$ takes value
\[
DM (G/H) = \Hom(M(G/H), \bQ).
\]
The conjugation maps for $M$ induce conjugation maps for $DM$, though
the contravariance of $D(-)$ requires us to use $C_{g^{-1}}$ for $M$ to define $C_g$ for $DM$.
The induction maps of $DM$ are induced from the restriction maps of $M$
and the restriction maps are induced from the induction maps of $M$.
\end{lem}

Many well-known structures arising from group theory can be assembled into Mackey functors.

\begin{ex}
Let $R(G)$ denote the ring of complex representations of the finite group $G$.
We define a rational Mackey functor $M_R$ by $M_R(G/H) = R(H) \otimes \bQ$,
with induction and restriction induced by induction and restriction of representations.

The ring structure on $R(G)$ gives more structure to this Mackey functor;
it is in fact a \emph{Tambara functor}. See Strickland \cite{stricktamb} for a survey of such functors
and related notions like Green functors.
\end{ex}

\begin{ex}
The zeroth equivariant stable homotopy groups of a $G$-spectrum form a Mackey functor.
For $X$ an orthogonal $G$-spectrum over a complete $G$-universe,
let $[-, X]^G \otimes \bQ$ denote the functor that sends $G/H$ to
\[
[\Sigma^\infty G/H_+, X]^G \otimes \bQ
\cong
[\Sigma^\infty S^0, X]^H \otimes \bQ
\cong
\pi_0^H(X)\otimes \bQ.
\]
We leave the induction, restriction and conjugation maps to the standard references
of May \cite[Chapter XIX]{may96} and Lewis et al.\ \cite[Section V.9]{lms86}.

We also note that $G$-equivariant cohomology theories use Mackey functors
as their coefficients, rather than abelian groups.
\end{ex}

\begin{ex}\label{ex:fixedpointmackeyfunctors}
Given a $\bQ[G]$-module $V$, we may define a rational Mackey functor $\mackeyfunctor{G}(V)$ (also called $FP_V$) as taking value $V^H$ at $G/H$.
The restriction maps are inclusion of fixed points and the induction maps are given by coset orbits.

We could also define a Mackey functor $FQ_V$ by taking value $V/H$ at $G/H$, with induction maps 
the quotient maps and restriction given by summing over a coset. 
The two functors are related via duality, see \cite[Proposition 4.1]{tw95}.

In the rational case the values $V^H$ and $V/H$  are isomorphic, as we now explain.
Since $G$ is finite, there is a diagram
\[
\xymatrix@C+0.5cm{
V^H \ar[r]_{\textrm{inclusion}} &
V \ar[r]_{\textrm{quotient}} \ar@/_1pc/[l]_{\av_H}   &
V/H  \ar@/_1pc/[l]_{\av'_H}
}
\]
where
\[
\av_H(x) = \frac{1}{|H|} \sum_{h \in H} hx
\quad \textrm{and} \quad
\av'_H([x]) = \frac{1}{|H|} \sum_{h \in H} hx .
\]
The composite
of inclusion and quotient $V^H \cong V/H$
is an isomorphism with inverse given by
the composite $\av_H \circ \av_H'$.
\end{ex}
When $V=\bQ$ with trivial $G$-action, $\mackeyfunctor{G}(\bQ)$ is an instance of
the constant Mackey functor, see Example \ref{ex:constantMackey}.

\begin{ex}\label{ex:burnsideringasmackey}
The rational Burnside rings for subgroups of $G$ assemble into a Mackey functor,
$\burnsidering(G/H)=\burnsidering(H)$,
the rational Grothendieck ring of finite $H$-sets, see for example \cite[Section 1.2]{tomdieck1} for details. 
The structure maps of this Mackey functor are the usual
restriction and induction of sets with group actions.
Moreover, the restriction maps are maps of rings.
See Examples \ref{ex:c6burnside} and~\ref{ex:c6burnsidemackey}
for worked examples in the case of $G=C_6$. 

\end{ex}

As is well-known, the rational Burnside ring splits, which is an immediate consequence of the following result.
\begin{lem}[tom Dieck's Isomorphism]\label{lem:IsoBurnsideFinite}
For $G$ a finite group, there is an isomorphism of rings
\[
\burnsidering(G) \lra C(\mathrm{Sub}(G)/G, \bQ) = C(\mathrm{Sub}(G), \bQ)^G = \prod_{(H) \leqslant G} \bQ
\]
where $\mathrm{Sub}(G)/G$ is the set of conjugacy classes of subgroups
of $G$ and $C(\mathrm{Sub}(G)/G, \bQ)$ is the set of continuous maps
between the two spaces (both equipped with the discrete topology). Here $(H)$ denotes the conjugacy class of $H$ in $G$.
We define $C(\mathrm{Sub}(G), \bQ)$ to have a $G$-action
by conjugation on the domain.

We define $e_H^G \in \burnsidering(G)$ to be the element of the Burnside ring corresponding to
the characteristic function of $(H)$ in $C(\mathrm{Sub}(G)/G, \bQ)$: the function that sends $(H)$ to $1 \in \bQ$ and all the other points to $0$. We may omit the superscript $G$ from the notation on idempotents when the context is clear. 
\end{lem}
\begin{proof}
The isomorphism is defined by sending a $G$-set $T$ to the map $(H) \mapsto |T^H|$.
Since the domain and codomain have the same dimension, the result follows from
proving the map is surjective, which follows from the formulas of the following lemma.
\end{proof}

Lemma \ref{lem:IsoBurnsideFinite} describes idempotents of the rational Burnside ring of $G$ in a simple, but more abstract way. It is often useful to write these idempotents in terms of the additive basis of the Burnside ring. The formula is given by Gluck \cite[Section 3]{gluck}.
\begin{lem}\label{lem:idempotentformula}
For $H$ a subgroup of $G$, the idempotent $e_H^G \in \burnsidering(G)$ is given by the formula
\[
e_H^G =  \sum_{K \leqslant H} \frac{|K|}{|N_G H|} \mu(K,H) G/K
\]
where $\mu(K,H) = \Sigma_i (-1)^i c_i$ for $c_i$ the number of
strictly increasing chains of subgroups from $K$ to $H$ of length $i$.
The length of a chain is one less than the number of subgroups involved
and $\mu(H,H)=1$ for all $H \leqslant G$.

For $H$ a subgroup of $G$, the set $G/H$ can be expressed as a sum of idempotents
\[
G/H = \sum_{K \leqslant H} \frac{|N_G K|}{|H|} e_K^G.
\]
\end{lem}

\begin{ex}
Let $G=C_2$, the ring $\burnsidering(G)$ is additively generated by the one-point space $1=C_2/C_2$
which is the monoidal unit, and $C_1/\trivgp$. 
The only non-evident multiplication is
\[
C_1/\trivgp \times C_1/\trivgp = 2 C_1/\trivgp.
\]
It follows that $e_1 = (1/2) C_2$ is an idempotent, as is $e_{C_2} = 1- e_{1}$.
Looking at the fixed points of these sets show that the idempotents are correctly named
and we recover the isomorphism
\[
\burnsidering(C_2) \cong  \bQ \langle e_1 \rangle \times \bQ \langle e_{C_2} \rangle.
\]
See Example \ref{ex:c6burnside} for the case of $G=C_6$.

\end{ex}

\begin{rmk}\label{rmk:restrictidempotents}
The restriction map $\burnsidering(H) \to \burnsidering(K)$ in terms of
\[
C(\mathrm{Sub}(H)/H, \bQ) \to C(\mathrm{Sub}(K)/K, \bQ)
\]
corresponds to precomposing with the map including subgroups
$\mathrm{Sub}(K) \to \mathrm{Sub}(H)$ and taking suitable orbits.
We can use this description to see how the restriction map interacts with idempotents.
For $A$ and $H$ subgroups of $G$, the restriction of the idempotent
$e_H^G$ to $A$ is still an idempotent, but it is not always $e_H^A$.
Instead,
\[
R_A^G (e_H^G) =   \sum_{\substack{K \leqslant_A A \\ K \in (H)_G }} e_K^A
\]
where the sum runs over $A$-conjugacy classes of subgroups $K$ of $A$, such that
$K$ is $G$-conjugate to $H$.

We see that if $H$ is not $G$-subconjugate to $A$, this will be zero.
Contrastingly, if $H$ is $G$-conjugate to $A$, then the only term in the summand
will be $K=A$ and $R_A^G (e_H^G) = e^A_A$.
\end{rmk}

Given a $G$-Mackey functor $M$, we can define an action of the Burnside ring $\burnsidering(H)$ on the
abelian group $M(G/H)$ by
\[
[H/K]: =I_{K}^{H} \circ  R_{K}^{H} \co M(G/H) \lra M(G/H)
\]
and extending linearly from the additive basis for $\burnsidering(H)$ given by
$H/K$ for subgroups $K$ of $H$.
The Mackey axiom implies that this action is compatible with the
multiplication of $\burnsidering(H)$, so that $M(G/H)$ is a module over
$\burnsidering(H)$. Moreover, the following square commutes.
\[
\xymatrix{
\burnsidering(G/H) \otimes M(G/H)
\ar[r] \ar[d]_{R_K^H \otimes R_K^H} &
M(G/H) \ar[d]_{R_K^H} \\
\burnsidering(G/K) \otimes M(G/K)
\ar[r] &
M(G/K)
}
\]
The action of Burnside rings is compatible with induction in the sense
of the \emph{Frobenius reciprocity} relations.
For $\alpha \in \burnsidering(G/H)$, $\beta \in \burnsidering(G/K)$,
$m \in M(G/H)$ and $n \in N(G/H)$
\[
\alpha \cdot I_K^H (m) = I_K^H ( R_K^H (\alpha) \cdot m)
\quad \quad
I_K^H(\beta) \cdot n = I_K^H ( \beta \cdot R_K^H(n))
\]
See \cite[Definition 2.3 and Example 2.11]{yosh80}.

\begin{lem}
Given an idempotent $e \in \burnsidering(G)$ and a $G$-Mackey functor $M$,
we can define a new Mackey functor $eM$ by
\[
(eM)(G/H) = R_H^G(e) M(G/H).
\]
\end{lem}
\begin{proof}
The conjugation and restriction maps are as for $M$, since these
actions are compatible with restriction.

By Frobenius reciprocity, the induction map for $K \leqslant H$ gives a map
\[
R_K^G(e) M(G/K)
\xrightarrow{I_K^H}
R_H^G(e) M(G/H). \qedhere
\]
\end{proof}

\section{Change of group functors}

As one should expect, we have adjunctions
coming from inclusions of subgroups and projections onto quotients.

\begin{defn}
Given an inclusion of a subgroup $i \co H \to G$, there
are functors
\[
\preextend \co \mackey{G} \lra \mackey{H}
\quad
\text{and}
\quad
\preforget \co \mackey{H} \lra \mackey{G}.
\]
Using the extension of Mackey functors to finite $G$-sets,
we may define the functor $\preforget$ as pre-composition with the
forgetful functor on sets with group actions.
The functor $\preextend$ is defined by pre-composition with extension of groups.
Thus for $M \in \mackey{G}$, $N \in \mackey{H}$, $A$ a $G$-set and $B$ a $H$-set,
\[
(\preextend M )(B) = M(G \times_H B)
\qquad \qquad
(\preforget N)(A) = N(i^* A).
\]
Similar definitions hold for the induction, restriction
and conjugation maps, and for morphisms of Mackey functors.
\end{defn}

\begin{lem}
Given an inclusion of a subgroup $i \co H \to G$, there is an adjunction
\[
\adjunction{\preextend}{\mackey{G}}{\mackey{H}}{\preforget}
\]
with each functor both left and right adjoint to each other.
\end{lem}
\begin{proof}
To see that this is an adjunction with $\preextend$ as the left adjoint, we take a map $f \co M \to \preforget N$
and construct a map $\bar{f} \co \preextend M \to N$. Consider an $H$-set $B$,
the map $\bar{f}(B)$ is given by the composite
\[
M(G \times_H B) \xrightarrow{f(B)}
N(i^* (G \times_H B))
\xrightarrow{N(\eta_{B})}
N(B)
\]
where the second map is induced (by using restriction maps) from the canonical map of $H$-sets
$\eta_B \co B \lra i^*( G \times_H B)$.
Conversely, given $g \co \preextend M \to N$ we construct
$\hat{g} \co M \to \preforget N$ in a similar way.
Given a $G$-set $A$, $\hat{g}(A)$ is the composite
\[
M(A) \xrightarrow{{M(\varepsilon_A)}}
M(G \times_H i^* A)  \xrightarrow{g(i^* A)} N(i^* A)
\]
where the first map is induced (by using restriction maps) from
$\varepsilon_A \co G \times_H i^* A \lra A$.

Now we take a map $f \co M \to \preforget N$ and show that it is equal to
$\hat{\bar{f}} \co M \to \preforget N$ (the other case of $\bar{\hat{g}}=g$  is similar).
The map $\hat{\bar{f}}$ is defined by taking the lower path in the following diagram.
\[
\xymatrix@C+1cm{
M(A) \ar[r]^{f(A)} \ar[d]_{M(\varepsilon_A)}
&
N(i^* A) \\
M(G \times_H i^* A)  \ar[r]^{f(G \times_H i^* A)}
&
N(i^* G \times_H i^* A)
\ar[u]_{N(\eta_{i^* A})}
}
\]
That we have an adjunction  follows as
\[
N(\eta_{i^* A}) \circ f(G \times_H i^* A) \circ M(\varepsilon_A)
=
N(\eta_{i^* A}) \circ N(i^* \varepsilon_A) \circ  f(A)
=
f(A)
\]
by the triangle identity for sets with group actions.

The proof that $(\preforget , \preextend)$ is an adjunction is very similar to the
previous case. The primary difference is that one uses
induction maps rather than restriction maps.
\end{proof}

We want to reproduce this construction for a quotient $\varepsilon \co G \to G/N$.
To make an adjunction, we need to restrict the category of $G$-Mackey functors somewhat.
We take a strong restriction, so that the two functors we produce will be both left and right adjoint
to each other.

\begin{defn}
For $N$ a normal subgroup of $G$, the category
$\mackey{G}/N$ is the full subcategory of $\mackey{G}$ of Mackey functors that are trivial on
those $G/K$ where $K$ does not contain $N$.
\end{defn}

\begin{defn}
Given a quotient map $\varepsilon \co G \to G/N$ for $N$ a normal subgroup of $G$, there
are functors
\[
\preinflate  \co \mackey{G}/N \lra \mackey{G/N}
\quad
\text{and}
\quad
\otherfunctor  \co \mackey{G/N} \lra \mackey{G}/N.
\]
Thus for $M \in \mackey{G}/N$, $M' \in \mackey{G/N}$, $K$ a subgroup of
$G$ containing $N$ and $B$ a $G/N$-set,
we define
\[
\preinflate  M(B) = M(\varepsilon^* B)
\quad \text{and} \quad
\otherfunctor  M'(G/K) = M'((G/N)/(K/N) ).
\]
If $K$ does not contain $N$ we set $\otherfunctor  M'(G/K)=0$.

The structure maps of $M$ and $M'$ are defined in terms of these formulas, as
are maps of Mackey functors.
\end{defn}

\begin{lem}
Given $N$ a normal subgroup of $G$, there is an adjunction
\[
\adjunction{\preinflate }{\mackey{G}/N}{\mackey{G/N}}{\otherfunctor }
\]
with each functor both left and right adjoint to each other.
\end{lem}
\begin{proof}
Both cases are similar and use the fact that
\[
\varepsilon^* (G/N)/(K/N) = G/K.
\]
We give one part of the proof as an illustration.

Take $f \co M \lra \otherfunctor  M'$ a map of $G$-Mackey functors that are trivial on
those $G/K$ where $K$ does not contain $N$.
We want to construct $\bar{f} \co \preinflate  M \lra  M'$.
Take a subgroup $K/N$ of $G/N$, we define
\[
\bar{f}((G/N)/(K/N)) = f(G/K) \co M(G/K) \to M'((G/N)/(K/N)). \qedhere
\]
\end{proof}

We give one more adjunction, between the category of rational $G$-Mackey functors and $\bQ$-modules
with an action of $G$.
\begin{lem}
There is an adjunction
\[
\adjunction{(-)(G/e)}{\mackey{G}}{\bQ[G] \leftmod}{\mackeyfunctor{G}}
\]
with each functor both left and right adjoint to each other.

The functor $(-)(G/e)$ sends a $G$-Mackey functor to the value $M(G/e)$.
Its adjoint $\mackeyfunctor{G}$ is defined in Example \ref{ex:fixedpointmackeyfunctors}
and at $G/H$ takes value $V^H$.
\end{lem}
\begin{proof}
Take a map $f \co M \to \mackeyfunctor{G}(V)$.
Evaluating at $G/e$ gives a map
$\bar{f} \co M(G/e) \to V$.
In the other direction, one starts with a map
$g \co M(G/e) \to V$ of $\bQ[G]$-modules.
The restriction map $R_e^H \co M(G/H) \to M(G/e)$
takes values in $M(G/e)^H$ as conjugation by elements of $H$ is trivial
in $M(G/e)$. We define
$\hat{g}$ as $f(G/e)^H \circ R_e^H$.

For the adjunction in the other direction, we
use the isomorphic description of $\mackeyfunctor{G}(V)(G/H)$ in terms of
$V/H$ and follow a similar pattern, using the induction maps of $M$ to define the
adjoint of a map $V \to M(G/e)$.
\end{proof}

\section{The classification of rational Mackey functors}\label{sec:classification}

Let $e_H^G \in \burnsidering(G) = \prod_{(H) \leqslant G} \bQ$
be the idempotent that is 1 on factor $H$ and zero elsewhere.
As described above, we can form a full subcategory of $\mackey{G}$ consisting of those
Mackey functors of the form $e_H^G M$.
Applying $e_H^G$ defines a functor $\mackey{G} \lra e_H^G \mackey{G}$.
It follows that we have a splitting
\[
\mackey{G} \cong \prod_{H \leqslant_G G} e_H^G \mackey{G}.
\]
To classify rational Mackey functors, it therefore suffices to classify the categories
$e_H^G \mackey{G}$. The key step is the following theorem giving a sequence of adjunctions.
The proof of the theorem occupies the rest of this section.

\begin{thm}\label{thm:manyadjunctions}
For $H \leqslant G$, there is a sequence of adjunctions of exact functors.
\[
\xymatrix@C+0.2cm{
e_H^G \mackey{G}
\ar@<0.7ex>[r]^-{\preextend} &
\ar@<0.7ex>[l]^-{\preforget }
R_{N_G H}^G (e_H^G)  \mackey{N_G H}
\ar@<0.7ex>[r]^-{\preinflate } &
\ar@<0.7ex>[l]^-{\otherfunctor }
\mackey{N_G H/H}
\ar@<0.7ex>[r]^-{(-)(W_G H/e)} &
\ar@<0.7ex>[l]^-{\mackeyfunctor{W_G H}}
\mathbb{Q}[W_G H] \leftmod
}
\]
with each pair both left and right adjoint to each other.
\end{thm}

\begin{lem}\label{lem:restrictedforgetful}
The adjunction $(\preextend, \preforget )$ restricts to an adjunction
\[
\adjunction{\preextend}{e_H^G \mackey{G}}{R_{N_G H}^G (e_H^G) \mackey{N_G H}}{\preforget }.
\]
The functors are exact and are both left and right adjoint to each other.
\end{lem}
\begin{proof}
Take $M \in e_H^G \mackey{G}$ and $K \leqslant N_G H$.
Since $M = e_H^G M$, we have the first equality below
\[
\begin{array}{rcl}
(\preextend  M) (N_G H /K)
&=&  R_K^G (e_H^G) M(G/K)  \\
&=& R_K^{N_G H} \circ  R_{N_G H}^G (e_H^G)  M(G/K) \\
&=& \left( R_{N_G H}^G (e_H^G)   (\preextend  M) \right ) (N_G H /K).
\end{array}
\]
Thus $(\preextend  M) \in R_{N_G H}^G (e_H^G) \mackey{N_G H}$.

Conversely, let $M' \in R_{N_G H}^G (e_H^G) \mackey{N_G H}$ and $K \leqslant G$.
Then
\[
(\preforget M') (G /K)
=  M'(i^* G/K)
= \oplus_{\lambda \in \Lambda} M'(N_G H/L_\lambda)
\]
where $i^* G / K$ decomposes as $\coprod_{\lambda \in \Lambda} N_G H/L_\lambda$.
Since
\[
M'(N_G H/L_\lambda)
=
R_{L_\lambda}^{N_G H} \circ R_{N_G H}^G (e_H^G) M'(N_G H/L_\lambda)
=
R_{L_\lambda}^G (e_H^G) M'(N_G H/L_\lambda),
\]
it follows that $\preforget M' = e_H^G \preforget M'$.

The functors are additive and left and right adjoint to each other.
Hence they are exact.
\end{proof}

\begin{lem}\label{lem:inflationidempotent}
For any $\mathbb{Q}[W_G H]$-module $V$, there is a canonical isomorphism
of $N_G H$-Mackey functors
\[
R_{N_G H}^G(e_H^G) \otherfunctor  \mackeyfunctor{W_G H}(V) \cong
\otherfunctor  \mackeyfunctor{W_G H}(V).
\]
It follows that we have an adjunction
\[
\xymatrix{
R_{N_G H}^G(e_H^G) \mackey{N_G H}
\ar@<0.7ex>[r]
&
\ar@<0.7ex>[l]
\bQ[W_G H] \leftmod
}
\]
with the functors both left and right adjoint to each other.
\end{lem}
\begin{proof}
The inclusion of an idempotent summand gives the map.
To see that this inclusion is an isomorphism, we evaluate both sides
a subgroup $A \leqslant N_G H$ that contains $H$.
Both domain and codomain take value zero on
subgroups that do not contain $H$.

We first decompose $R_A^G(e_H^G)$ into idempotents of the
rational Burnside ring of $A$
\[
R_A^G (e_H^G) =   \sum_{\substack{K \leqslant_A A \\ K \in (H)_G }} e_K^A.
\]
Secondly, by Lemma \ref{lem:idempotentformula} $(e_K^A)$ is a sum of
$|W_A K|^{-1} [A/K]$ and rational multiples of basis elements
$[A/K']$, for $K'$ a proper subgroup of $K$.
The element $[A/L]$ acts on $\otherfunctor  \mackeyfunctor{W_G H}(V)(N_G H/A)$
through $I^A_L \circ R^A_L$. This is zero unless $L$ contains $H$.
Hence each $[A/K']$ acts as zero, and $[A/K]$ only acts non-trivially when
$K$ contains $H$. Since $K$ is also $G$-conjugate to $H$, we see that
$K=H$. Hence,
\[
\Big(
\sum_{\substack{K \leqslant_A A \\ K \in (H)_G }} e_K^A
\Big)
\otherfunctor  \mackeyfunctor{W_G H}(V)(N_G H/A)
=
e_H^A \otherfunctor  \mackeyfunctor{W_G H}(V)(N_G H/A)
=
e_H^A V^{A/H}
\]
with $e_H^A$ acting through $|W_A H|^{-1} I^A_H \circ R^A_H$.

Thirdly, $I^A_H \circ R^A_H$ is the composite
\[
V^A \lra V^H \lra V^A
\]
with the first map the inclusion and the second map taking the sum over
$A/H$-coset representatives. Hence this map is multiplication by $|A/H|$.
Since $H$ is normal in $A$, it follows that
$|W_A H|^{-1} I^A_H \circ R^A_H$ acts through the identity, giving the first statement.

For the second statement, the inclusion of the full subcategory
\[
R_{N_G H}^G (e_H^G)  \mackey{N_G H}
\lra
\mackey{N_G H}/H
\]
has an adjoint, which is applying the idempotent $R_{N_G H}^G (e_H^G)$.
This adjoint is both left and right adjoint to the inclusion.

Composing this with the adjunctions $(\preinflate ,\otherfunctor )$
and $((-)(W_G H/e) , \mackeyfunctor{W_G H})$ gives the result.
The functors in each adjunction are additive,
and are left and right adjoint to each other.
Hence they are exact.
\end{proof}

\begin{defn}
For $H \leqslant G$, define functors
\[
\begin{array}{l}
F_H \co \mathbb{Q}[W_G H] \leftmod \lra e_H^G \mackey{G} \\
U_H \co e_H^G \mackey{G} \lra \mathbb{Q}[W_G H] \leftmod
\end{array}
\]
where $F_H$ is the composite of the lower level functors from the diagram in Theorem \ref{thm:manyadjunctions}
and $U_H$ is the composite of the the upper level functors.
\end{defn}

We see immediately that the additive functors
$F_H$ and $U_H$ are both left and right adjoint
to each other and that $U_H M  = M(G/H)$.

\begin{prop}\label{prop:FHformula}
For $K \leqslant G$,
\[
F_H (V)(G/K) = (\mathbb{Q}[(G/K)^H] \otimes V)^{W_G H}.
\]
Moreover, the Mackey functor $F_H (V)$ is both projective and injective, and
$e_H^G F_H (V) =F_H (V)$.
\end{prop}
\begin{proof}
From the definitions, the composite is given by
\[
G/K \mapsto \underset{\lambda \in \Lambda}{\bigoplus} V^{L_\lambda/H}
\]
where $G / K$ decomposes as $\coprod_{\lambda \in \Lambda} N_G H/L_\lambda$
and $K$ contains $H$. If $K$ does not contain $H$, the composite takes value zero.
Each factor in this decomposition corresponds to
an $N_G H$-orbit in the set of $N_G H$-maps $N_G H/ H \to i^*G/K$. The $N_G H$-action is by right multiplication by the inverse on $N_G H/H$. Such a map corresponds to a $G$-map
$G/H \to G/K$, which is simply an element $\alpha$ of the set $(G/K)^H$.
By thinking of $(G/K)^H$ as a $W_G H$-set, we can sum over all $\alpha$ to obtain the formula
\[
G/K \mapsto \Big( \underset{\alpha \in (G/K)^H}{\bigoplus} V_\alpha \Big)^{W_G H}
\]
with $W_G H$ permuting the summands (it acts by right multiplication by
the inverse on $(G/K)^H$). Replacing summands by a tensor product gives the formula
\[
F_H (V)(G/K) = (\mathbb{Q}[(G/K)^H] \otimes V)^{W_G H}.
\]

Every $\bQ[W_G H]$-module is both injective and projective.
Hence, $F_H (V)$ is both projective and injective as
the functors $F_H$ and $U_H$ are exact.

Lemmas \ref{lem:restrictedforgetful}
and \ref{lem:inflationidempotent} give the statement about idempotents.
\end{proof}

This new formula is very compact, but for calculations 
(see Examples \ref{ex:freec6functors} and \ref{ex:freeongroupring}) 
and later theory we 
will also need simple description of the induction and restriction maps of $F_H (V)$.

\begin{lem}\label{lem:freeinductionrestriction}
For $L \leqslant  K \leqslant G$, the induction map
\[
(\mathbb{Q}[(G/L)^H] \otimes V)^{W_G H} =
F_H (V)(G/L) \lra F_H (V)(G/K) =
(\mathbb{Q}[(G/K)^H] \otimes V)^{W_G H}
\]
is induced by the projection $\alpha:G/L \to G/K$.

The restriction map
\[
F_H (V)(G/K) \lra F_H (V)(G/L)
\]
is induced by the map $\bQ[(G/K)^H] \to \bQ[(G/L)^H]$
that sends $gK$ to the sum of the elements in its preimage
under the projection $\alpha:G/L \to G/K$.
\end{lem}
\begin{proof}
Write $(\mathbb{Q}[(G/L)^H] \otimes V)^{W_G H}$ as
\[
\Big( \underset{\sigma \co G/H \to G/L}{\bigoplus} V_{\sigma} \Big)^{W_G H},
\]
we use the subscript $\sigma$ on $V$ to keep track of the factors.
The $W_G H$-action is given by both acting on $V$ and permuting the summands.
That is, for $w \in W_G H$ and $v \in F_H (V)(G/L)$, we define
$wv$ to have component in summand $\sigma$ given by
\[
g (v_{\sigma \circ r(w^{-1})})
\]
where $r(w^{-1}) \co G/H \lra G/H$ is right multiplication by $w^{-1}$.

Chasing through the definitions, it follows that the restriction map
is given by
\[
(R_L^K y)_{\sigma}  = y_{\alpha \circ \sigma}.
\]
The induction map is given by
\[
(I_L^K y)_{\tau}  = \sum_{\alpha \circ \sigma = \tau} y_{\sigma}
\]
the sum over those summands $\sigma$ that map to $\tau$ by $\alpha$.
\end{proof}

\begin{thm}\label{thm:mainsplitting}
For each $H \leqslant G$ there is an equivalence of categories
\[
\adjunction{U_H}{e_H^G \mackey{G}}{\bQ[W_G H] \leftmod}{F_H}
\]
Hence there is an equivalence of categories
\[
\mackey{G} \cong \prod_{(H) \leqslant G} \bQ[W_G H] \leftmod
\]
where the product runs over $G$-conjugacy classes of subgroups of $G$.
\end{thm}
\begin{proof}
We have already seen that $U_H$ and $F_H$ are both left and right adjoint to each other.
The unit is an isomorphism:
\[
V \lra U_H F_H V = (\mathbb{Q}[(G/H)^H] \otimes V)^{W_G H} \cong V.
\]
It follows that the counit is an isomorphism of Mackey functors of the form $F_H V$.

The rest of the proof shows that any Mackey functor is a finite direct sum of
Mackey functors of the form $F_H V$ for varying $H$ and $V$.

We partition the set of subgroups of $G$ into sets, which we may think
of as their height in the subgroup lattice.
We start with $S_0 = \{ e \}$, then we define
$S_j$ as those groups not in $S_{j-1}$ but all of whose subgroups
are in $S_i$ for $i < j$.
Each $S_j$ is closed under conjugation, with $n_j$ conjugacy classes.
Choose a $H_{j,k}$ in each conjugacy class, $1 \leqslant k \leqslant n_j$.
We say that a Mackey functor $M$ is of type $(j,k)$ if
$M(G/H_{j,k})$ is non-zero, but
\[
M(G/H_{j',k'})=0 \quad \text{ for } j'< j \text{ and for } j'=j \text{ and } k'<k.
\]

We argue via descending induction.
Starting at the top, if $M(G/H)=0$ for all proper subgroups $H$, then
$M = F_G M(G/G)$.
Fix $(j,k)$ inductively and assume that all Mackey functors of type $(j', k')$
for $j'> j$ and for $j'=j$ and $k'>k$ are finite direct sums of 
Mackey functors of the form $F_J V_J$ where $V_J$ is a $W_G J$-module
and
\[
J \in \{  H_{j'', k''} \mid
j'' > j' \text{ or } j'' = j' \text{ and } k'' \geqslant k'
\}.
\]
Let $M$ be a Mackey functor of type $(j,k)$.
For $H= H_{j,k}$,
there is a map of Mackey functors
\[
\kappa \co M \lra F_H M(G/H) = e_H^G F_H  M(G/H)
\]
that is the identity on $G/H$.
The kernel and cokernel of $\kappa$ are, by inductive assumption,
of the form $F_J V_J$. It follows that $e_H^G$ applied to the
kernel and cokernel are zero.
Thus $e_H^G \kappa$ is an isomorphism and so $\kappa$,
which is equal to the epimorphism $M \to e_H^G M$ followed by $e_H^G \kappa$,
is an epimorphism.

Since $F_H M(G/H)$ is projective, the epimorphism splits and
$M$ is a direct sum of $F_H M(G/H)$ and Mackey functors
of the form $F_J V_J$.
\end{proof}

\begin{cor}
Every rational Mackey functor is both projective and injective.
\end{cor}

By Remark \ref{rmk:restrictidempotents}, for $H \leqslant G$
we have $R^G_H (e^G_H) = e^H_H$. This gives the following corollary of
Theorem \ref{thm:mainsplitting}

\begin{cor}\label{cor:determinesmackey}
A $G$-Mackey functor $M$ is uniquely determined by the
collection
\[
\{ e_H^H M(G/H) \in \bQ[W_G H] \leftmod \ | \ (H) \leqslant G \}
\]
where we index over $G$-conjugacy classes of subgroups of $G$.
\end{cor}

The remaining question is how to conveniently find the values $M(G/H)$
of the Mackey functor $M$ from such a collection.
The next section gives a formula that provides a satisfying answer.

\section{The diagonal decomposition}

Rational Mackey functors for compact Lie groups are
considered in Greenlees \cite{greratmack}.
We present the decomposition formula for rational $G$-Mackey functors
for finite $G$, following Examples C i) and Corollary 5.3 of \cite{greratmack}.
The reference proves the result using equivariant stable homotopy theory;
a direct algebraic proof is given by Sugrue \cite[Lemma 6.1.9]{sugruethesis}.
We use the structure results to prove it via a calculation on
Mackey functors of the form $F_A V$.

\begin{thm}\label{thm:mackeydecomp}
Let $M$ be a $G$-Mackey functor and let $K \leqslant H$ be subgroups of $G$.
Then
\[
e_K^H M(G/H)
\cong \big( e_K^K M(G/K) \big)^{W_H K}.
\]
\end{thm}

\begin{proof}
By Theorem \ref{thm:mainsplitting}, we may assume $M$ is of the form
$F_A V$, for $V$ a $\bQ[W_G A]$-module.
Lemma \ref{lem:freefunctoridempotent} implies that we
only need to consider the case $A=K$.

By Remark \ref{rmk:restrictidempotents} and Proposition \ref{prop:FHformula},
we may remove the idempotent $e_K^K$ from the formula.
Thus we must prove
\[
e_K^H (\bQ[(G/H)^K] \otimes V)^{W_G K}
=
e_K^H F_K(G/H)
\cong
\big( F_K(G/K) \big)^{W_H K}
=
\big( (\bQ[(G/K)^K] \otimes V)^{W_G K} \big)^{W_H K}.
\]

By Lemma \ref{lem:idempotentformula}, $(e_K^H)$ is a sum of
$|W_H K|^{-1} [H/K]$ and rational multiples of basis elements
$[H/L]$ for $L$ a proper subgroup of $K$.
An element $[H/L]$ acts on $F_K(G/H)$
through $I^H_L \circ R^H_L$.
As $F_K (V)(G/L)=0$ unless $L$ contains $K$ (up to $H$-conjugacy),
$(e_K^H)$ acts as 
\[
|W_H K|^{-1} [H/K] = |W_H K|^{-1} I_K^H \circ R_K^H.
\]

Lemma \ref{lem:freeinductionrestriction} shows how the map $I_K^H \circ R_K^H$ 
is induced from a composite 
\[
\bQ[(G/H)^K] \lra \bQ[(G/K)^K] \lra \bQ[(G/H)^K]. \tag^*{$\ast$}
\]
We look at this composite in detail. The reader may find it useful to 
compare the following argument with Example \ref{ex:idempotentdecomp}.

If an element $gK$ is fixed by left multiplication by elements of $K$
(that is, an element of $N_G K$), then $gH$ is also $K$-fixed.
Take $gK$ and $g'K$ that are $K$-fixed with $gH=g'H$.
Then $g' = gh$ for some $h \in H$, and
\[
K= (g')^{-1} K g' = (gh)^{-1} K gh = h^{-1} K h
\]
so $h \in N_H K$. It follows that the second map of ($\ast$) passes to an injection 
\[
\bQ[(G/K)^K]/W_H K \lra \bQ[(G/H)^K].
\]

Take $gH$ in the image of $(G/K)^K \to (G/H)^K$, the composite ($\ast$) sends this
to the sum of those $aH$ such that $aK = gK$. By the previous argument, this sum is
$|W_H K| gH$.
Now take a coset $gH$ that is not in the image of $(G/K)^K \to (G/H)^K$,
then the first map of ($\ast$) sends this to zero.

It follows that the composite 
\[
(\bQ[(G/H)^K] \otimes V)^{W_G K}
\lra
(\bQ[(G/K)^K] \otimes V)^{W_G K}
\lra
(\bQ[(G/H)^K] \otimes V)^{W_G K}
\]
induced from ($\ast$) in turn induces an isomorphism from
\[
\Big( (\bQ[(G/K)^K] \otimes V)^{W_G K} \Big)^{W_H K}
\]
onto the image of $I_K^H \circ R_K^H$ in $(\bQ[(G/H)^K] \otimes V)^{W_G K}$.
\end{proof}

\begin{lem}\label{lem:freefunctoridempotent}
Let $A$, $B$ and $C$ be subgroups of $G$, with $C \leqslant B$ and
let $V$ be a $\bQ[W_G A]$-module. Then
\[
e^B_C \big(  F_A (V)  (G/B) \big)= 0
\]
unless $A$ and $C$ are $G$-conjugate.
\end{lem}
\begin{proof}
By Proposition \ref{prop:FHformula}, $e_A^G F_A (V) =F_A (V)$.
Hence,
\[
e^B_C \big( F_A (V)  (G/B) \big)=
e^B_C R^G_B (e^G_A) \big( F_A (V)  (G/B) \big).
\]
The product $e^B_C R^G_B (e^G_A)$ is zero unless
$A$ is $G$-conjugate to a subgroup $A'$ of $G$ and that subgroup $A'$
is $B$-conjugate to $C$.
We also require that $A$ is $G$-subconjugate to $B$, as otherwise
$\bQ[(G/B)^A]=0$.
This is equivalent to requiring that $A$ is $G$-conjugate to $C$.
\end{proof}

We illustrate this decomposition with two examples.

\begin{ex}
Let $M$ be a rational Mackey functor for $C_{p^3}$.
Define $\bQ$-modules
\[
V_0 = M(C_{p^3}/C_1), \quad
V_1 = e_{C_{p^1}} M(C_{p^3}/C_{p^1}), \quad
V_2 = e_{C_{p^2}}M(C_{p^3}/C_{p^2}), \quad
V_3 = e_{C_{p^3}}M(C_{p^3}/C_{p^3}).
\]
where $e_{C_{p^i}} \in \burnsidering(C_{p^i})$ is the idempotent with support
$C_{p^i}$, $i \in \{0,1,2,3\}$.
Note that for $i=0$ this is $1 \in \burnsidering(C_1) = \bQ$.
The $\bQ$-module $V_i$ has an action of
$\bQ[W_{C_{p^3}} C_{p^i}  ] = \bQ[C_{p^3}/C_{p^i}]$.

The classification theorem implies that
\[
M \cong F_{e} V_0 \oplus F_{C_{p}} V_1  \oplus F_{C_{p^2}} V_2  \oplus F_{C_{p^3}} V_3.
\]
Writing out the values of $M$ at varying subgroups gives the diagram
\[
\xymatrix@C-0.8cm{
V_3 & \oplus &
V_2^{C_{p^3}/C_{p^2}} \ar@<-0.7ex>[d]
& \oplus &
V_1^{C_{p^3}/C_{p}} \ar@<-0.7ex>[d]
& \oplus &
V_0^{C_{p^3}/C_{1}} \ar@<-0.7ex>[d]
& = &
M(C_{p^3}/C_{p^3}) \ar@<-0.7ex>[d]
\\
&&
 V_2 \ar@<-0.7ex>[u]
& \oplus &
V_1^{C_{p^2}/C_{p}} \ar@<-0.7ex>[d] \ar@<-0.7ex>[u]
& \oplus &
V_0^{C_{p^2}/C_{1}} \ar@<-0.7ex>[d] \ar@<-0.7ex>[u]
& = &
M(C_{p^3}/C_{p^2}) \ar@<-0.7ex>[d] \ar@<-0.7ex>[u]
\\
&&
&&
V_1   \ar@<-0.7ex>[u]
& \oplus &
V_0^{C_{p}/C_{1}} \ar@<-0.7ex>[d] \ar@<-0.7ex>[u]
& = &
M(C_{p^3}/C_{p^1}) \ar@<-0.7ex>[d] \ar@<-0.7ex>[u]
\\
&&
&&
&&
 V_0 \ar@<-0.7ex>[u] & = & M(C_{p^3}/C_{1}) \ar@<-0.7ex>[u]
}
\]
With vertical maps indicating induction and restriction.
\end{ex}

\begin{rmk}
Given a $G$-Mackey functor $M$ and $H \leqslant G$,
we can construct $\overline{M}(G/H)$, the quotient of $M(G/H)$ by the images of the
induction maps from proper subgroups of $H$.
This example shows how $\overline{M}(G/H) = e^H_H M(G/H)$, so that the
classification result is based around stripping out the images of the induction functors.
\end{rmk}

\begin{ex}\label{ex:idempotentdecomp}
We look at a particular instance of Theorem \ref{thm:mackeydecomp} 
in order to illustrate the second half of the proof.
Let $G=S_4$ be the symmetric group on four letters, $K=\langle (12) \rangle $,
$H=\langle (12), (34) \rangle$. We have
\[
N_G K = H
\quad \quad
(G/K)^K = W_G K =W_H K= H/K = \{ K, (34) K \},
\quad \quad
(G/H)^K = \{ H, (14)(23) H \}.
\]
We consider the case of $F_K(\bQ[W_G K])$. We write out the two sides of 
Theorem \ref{thm:mackeydecomp} below. 
\[
\begin{array}{rcccl}
e_K^H F_K(\bQ[W_G K])(G/H)
& = &
e_K^H \bQ[(G/H)^K]
& = &
I_K^H \circ R_K^H \bQ[(G/H)^K] \\
\big( e_K^K F_K(\bQ[W_G K])(G/K) \big)^{W_H K} & \cong &  
(\bQ[(G/K)^K]^{W_H K} & =&  \bQ[\{ K + (34) K \}]
\end{array}
\]

To calculate $I_K^H \circ  R_K^H \bQ[(G/H)^K]$, 
consider the maps of Lemma \ref{lem:freeinductionrestriction}
\[
\bQ[\{ H, (14)(23) H \}]
\lra
\bQ[\{ K, (34) K \}]
\lra
\bQ[\{ H, (14)(23) H \}].
\]
The first map (restriction) sends $H$ to $K + (34) K$ and $(14)(23) H$ to zero.
The second map (induction) sends $K$ and $(34)K$ to $H$. 
The image of the composite is therefore the submodule 
of $\bQ[\{ H, (14)(23) H \}]$ generated by $H$.
Hence, $I_K^H \circ R_K^H \bQ[(G/H)^K]= \bQ[\{ H  \}]$
and the map induced by restriction
\[ 
e_K^H F_K(\bQ[W_G K])(G/H) = \bQ[\{ H  \}] \lra \bQ[\{ K + (34) K \}] \cong \big( e_K^K F_K(\bQ[W_G K])(G/K) \big)^{W_H K}
\]
is an isomorphism. 

\end{ex}

\section{Worked examples for \texorpdfstring{$G=C_6$}{G=C6}}\label{sec:c6examples}

Fixing the group to be $C_6$, the cyclic group of order 6, we give a detailed investigation of 
rational $C_6$-Mackey functors and their corresponding decomposition as modules over group rings. 
We choose a small group so that group-theoretic complexities do not obstruct the overall picture. 
We organise this section as a series of examples.

For clarity, write $C_6=\{0,1,2,3,4,5\}$, $C_3=\{0,2,4\}$, $C_2=\{ 0,3\}$, $C_1=\{0\}$
and quotients via underlining, so $C_6/C_3 = \{ \underline{0}, \underline{1} \}$.
We start by describing the rational Burnside ring of $C_6$ and its idempotents. \newline

\begin{example}\label{ex:c6burnside}
The rational Burnside ring of $C_6$ is additively generated by the sets $C_6$, $C_6/C_3$,
$C_6/C_2$ and $C_6/C_6$, which is the multiplicative unit. 
The multiplicative properties of these sets are:
\[
\begin{array}{ccc}
C_6 \times C_6 = 6 C_6 &
C_6 \times C_6/C_3 = 2 C_6 &
C_6 \times C_6/C_2 = 3 C_6  \\
C_6/C_3 \times C_6/C_3 = 2 C_6/C_3 &
C_6/C_2 \times C_6/C_3 = C_6   &
C_6/C_2 \times C_6/C_2 = 3 C_6/C_2 .
\end{array}
\]
Evidently, we have non-zero idempotents $C_6/C_6$, $(1/6)C_6$, $(1/2)C_6/C_3$, $(1/3)C_6/C_2$, but these
are not orthogonal to each other. Some algebraic manipulation gives an orthogonal decomposition
of the unit as 
\[
\begin{array}{rcl}
e_{C_1}^{C_6} &=& (1/6)C_6   \\
e_{C_2}^{C_6} &=& (1/3)C_6/C_2-(1/6)C_6/C_1   \\
e_{C_3}^{C_6} &=& (1/2)C_6/C_3-(1/6)C_6/C_1   \\
e_{C_6}^{C_6} &=& C_6/C_6 - (1/3)C_6/C_2 - (1/2)C_6/C_3 + (1/6)C_6.
\end{array}
\]
One can check that these formulas match those of Lemma \ref{lem:IsoBurnsideFinite} and 
verify the reverse formulas describing the homogeneous sets in terms of sums of idempotents. 
Furthermore, one can confirm that the map of Lemma \ref{lem:IsoBurnsideFinite}
is indeed an isomorphism in this case.
\end{example}

Now we have an understanding of the rational Burnside ring of $C_6$, we can look at the Burnside ring Mackey functor. 

\begin{example}\label{ex:c6burnsidemackey}
The rational Burnside ring Mackey functor of Example \ref{ex:burnsideringasmackey} is defined by 
\[
\burnsidering(C_6/H)=\burnsidering(H),
\]
with structure maps defined in terms of   
restriction and induction of sets with group actions. 
For example, $C_6/C_3 \in \burnsidering(C_6)$ restricts to $2 C_3/C_3 \in \burnsidering(C_3)$
and $C_2/C_1 \in \burnsidering(C_2)$ inducts to $C_6/C_1 \in \burnsidering(C_6)$. 

We can calculate the idempotents of the rational Burnside rings of $C_3$ and $C_2$ as we did for $C_6$.
\[
\begin{array}{ccc}
\burnsidering(C_3) & e_{C_1}^{C_3} = (1/3)C_3/C_1 & e_{C_3}^{C_3} = C_3/C_3 - (1/3)C_3/C_1 \\
\burnsidering(C_2) & e_{C_1}^{C_2} = (1/2)C_2/C_1 & e_{C_2}^{C_2} = C_2/C_2 - (1/2)C_2/C_1
\end{array}
\]

We present the rational Burnside ring Mackey functor as the following Lewis diagram, phrased in terms 
of idempotents. 
The downward arrows indicate that an idempotent restricts to another idempotent, 
where an idempotent restricts to zero, we omit an arrow. 
The upward arrows indicate that an idempotent inducts to a multiple of another idempotent. 

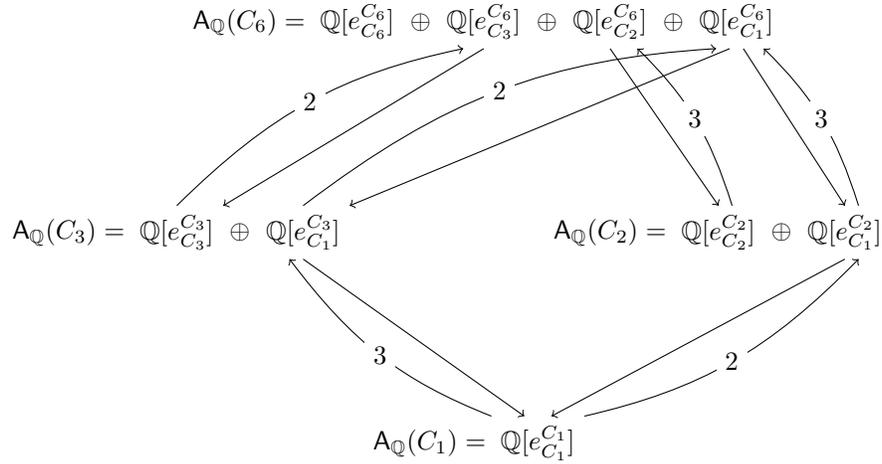
\begin{center}
\begin{tikzpicture}[xscale=1.2,yscale=1.4]
	 \tikzstyle{vertex}=[]
	 \tikzstyle{edge} = [draw,->,black]

	 \node[vertex] (a6) at (1+0.2,4) {$\burnsidering(C_6)=$};
	 \node[vertex] (c66) at (2.4,4) {$\bQ[e_{C_6}^{C_6}]$};
	 \node[vertex] at (3.1,4) {$\oplus$};
	 \node[vertex] (c63) at (3.8,4) {$\bQ[e_{C_3}^{C_6}]$};
	 \node[vertex] at (4.5,4) {$\oplus$};
	 \node[vertex] (c62) at (5.2,4) {$\bQ[e_{C_2}^{C_6}]$};
	 \node[vertex] at (5.9,4) {$\oplus$};
	 \node[vertex] (c61) at (6.6,4) {$\bQ[e_{C_1}^{C_6}]$};

	 \node[vertex] (a3) at (2-3+0.2,2) {$\burnsidering(C_3)=$};
	 \node[vertex] (c33) at (3.4-3,2) {$\bQ[e_{C_3}^{C_3}]$};
	 \node[vertex] at (4.1-3,2) {$\oplus$};
	 \node[vertex] (c31) at (4.8-3,2) {$\bQ[e_{C_1}^{C_3}]$};

	 \node[vertex] (a2) at (2+3+0.2,2) {$\burnsidering(C_2)=$};
	 \node[vertex] (c22) at (3.4+3,2) {$\bQ[e_{C_2}^{C_2}]$};
	 \node[vertex] at (4.1+3,2) {$\oplus$};
	 \node[vertex] (c21) at (4.8+3,2) {$\bQ[e_{C_1}^{C_2}]$};
	 
	 \node[vertex] (a1) at (3+0.2,0) {$\burnsidering(C_1)=$};
	 \node[vertex] (c11) at (4.4,0) {$\bQ[e_{C_1}^{C_1}]$};

 	 \draw[edge] ([xshift=-0.5ex]c61.south) to (c31.north east);
 	 \draw[edge] ([xshift=0.5ex]c61.south) to (c21.north);
 	 \draw[edge] (c63.south) to (c33.north east);
 	 \draw[edge] (c62.south) to (c22.north);
 	 \draw[edge] (c31.south) to ([xshift=-1ex]c11.north);
 	 \draw[edge] (c21.south) to ([xshift=1ex]c11.north);

  	 \draw[edge][bend left=15] (c11.north west) to node[fill=white]{3} ([xshift=-1ex]c31.south) ;
  	 \draw[edge][bend right=15] (c11.north east) to node[fill=white]{2} ([xshift=1ex]c21.south) ;
  	 \draw[edge][bend left=15] (c31.north) to node[fill=white]{2} ([xshift=-1.5ex]c61.south) ;
  	 \draw[edge][bend right=15] ([xshift=1ex]c21.north) to node[fill=white]{3} ([xshift=2ex]c61.south) ;
  	 \draw[edge][bend left=15] (c33.north) to node[fill=white]{2} ([xshift=-1.5ex]c63.south) ;
  	 \draw[edge][bend right=15] ([xshift=1ex]c22.north) to node[fill=white]{3} ([xshift=2ex]c62.south) ;

\end{tikzpicture}
\end{center}
\end{example}

As described at the start of Section \ref{sec:classification}, we may use the idempotents of 
$\burnsidering(C_6)$ to split the category of rational $C_6$-Mackey functors:
\[
\mackey{C_6} \cong 
e_{C_6}^{C_6} \mackey{C_6} \oplus 
e_{C_3}^{C_6} \mackey{C_6} \oplus 
e_{C_2}^{C_6} \mackey{C_6} \oplus 
e_{C_1}^{C_6} \mackey{C_6}. 
\]
Hence, any rational $C_6$-Mackey functor $M$ can be written as 
\[
M \cong 
e_{C_6}^{C_6} M \oplus 
e_{C_3}^{C_6} M \oplus 
e_{C_2}^{C_6} M \oplus 
e_{C_1}^{C_6} M. 
\]
Identifying the action of idempotents on a given Mackey functor $M$ 
from knowledge of the induction and restriction maps can be time-consuming.
It is often easier to produce examples of Mackey functors in the 
subcategories of the splitting. 
We do so using the equivalences of categories of 
Theorem \ref{thm:mainsplitting}:
\[
\begin{array}{c}
\adjunction{U_6}{e_{C_6}^{C_6} \mackey{C_6}}{\bQ[C_6/C_6] \leftmod}{F_6} \\
\adjunction{U_3}{e_{C_3}^{C_6} \mackey{C_6}}{\bQ[C_6/C_3] \leftmod}{F_3} \\
\adjunction{U_2}{e_{C_2}^{C_6} \mackey{C_6}}{\bQ[C_6/C_2] \leftmod}{F_2} \\
\adjunction{U_1}{e_{C_1}^{C_6} \mackey{C_6}}{\bQ[C_6/C_1] \leftmod}{F_1} \\
\end{array}
\]
where $U_H$ is evaluation at $G/H$ and 
\[
F_H (V)(G/K) = (\mathbb{Q}[(G/K)^H] \otimes V)^{W_G H}.
\]

\begin{example}\label{ex:freec6functors}
Take $\bQ \in \bQ[C_6/C_2] \leftmod$ (with trivial group action). 
We have 
\[
\begin{array}{rcl}
F_2 (\bQ) (C_6/C_6) &=& (\mathbb{Q}[(C_6/C_6)^{C_2}] \otimes \bQ)^{C_6/C_2}= \bQ \\
F_2 (\bQ) (C_6/C_3) &=& (\mathbb{Q}[(C_6/C_3)^{C_2}] \otimes \bQ)^{C_6/C_2}= 0 \\
F_2 (\bQ) (C_6/C_2) &=& (\mathbb{Q}[(C_6/C_2)^{C_2}] \otimes \bQ)^{C_6/C_2}= \bQ[C_6/C_2]^{C_6/C_2} = \bQ \\
F_2 (\bQ) (C_6/C_1) &=& (\mathbb{Q}[(C_6/C_1)^{C_2}] \otimes \bQ)^{C_6/C_2}= 0.
\end{array}
\]
The only non-trivial restriction map is $F_2 (\bQ) (C_6/C_6) \to F_2 (\bQ) (C_6/C_2)$.
By Lemma \ref{lem:freeinductionrestriction} this map is induced from the map
\[
\mathbb{Q}[(C_6/C_6)^{C_2}] \lra \mathbb{Q}[(C_6/C_2)^{C_2}]
\]
sending an element $q \underline{0}$ to $q (\underline{0}+ \underline{1}+ \underline{2})$. 
Hence, the restriction  map is the identity. 
It follows (either by using the Mackey axiom, or another application of Lemma \ref{lem:freeinductionrestriction})
that the non-trivial induction map is multiplication by 3 (the subgroup index).
We summarise this calculation, along with the analogous calculations for 
$F_1 (\bQ)$, $F_3 (\bQ)$ and $F_6 (\bQ)$, in the following Lewis diagrams.

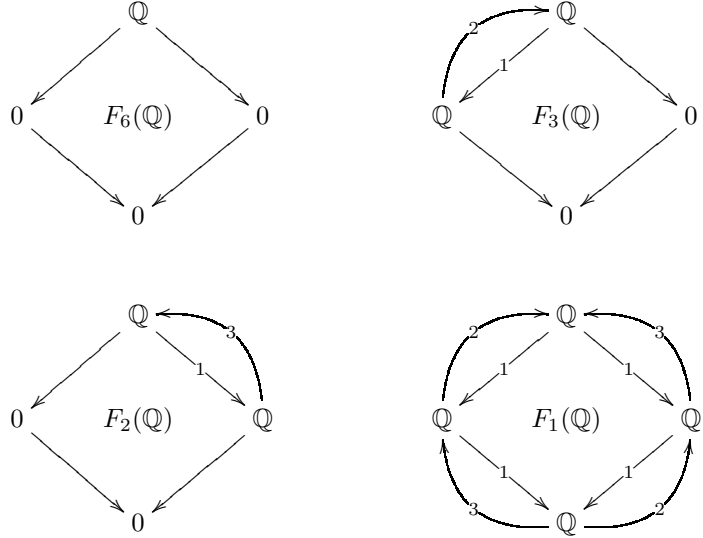
\[
\xymatrix{
& \bQ 
\ar[dl] \ar[dr]
&&& &
\bQ
\ar[dl]|1 \ar[dr] 
\\
0 \ar[dr]
& F_6 (\bQ) &
0 \ar[dl] 
& &
\bQ\ar[dr] \ar@/^1.5pc/[ur]|{2}
& F_3 (\bQ) &
0 \ar[dl] 
\\
& 0
&&&& 0 \\
& \bQ 
\ar[dl] \ar[dr]|1 
&&& &
\bQ
\ar[dl]|1 \ar[dr]|1 
\\
0 \ar[dr]
& F_2 (\bQ) &
\bQ \ar[dl] \ar@/_1.5pc/[ul]|{3}
& &
\bQ \ar[dr]|1 \ar@/^1.5pc/[ur]|{2}
& F_1 (\bQ) &
\bQ \ar[dl]|1 \ar@/_1.5pc/[ul]|{3}
\\
& 0
&&&& \bQ \ar@/_1.5pc/[ur]|{2} \ar@/^1.5pc/[ul]|{3}
} 
\]

We can relate this to Example \ref{ex:c6burnsidemackey}, where our idempotent calculations 
make the splitting evident. Indeed, 
\[
\burnsidering(C_6) \cong 
F_1 (\bQ) \oplus F_2(\bQ) \oplus F_3 (\bQ) \oplus F_6 (\bQ).
\]
\end{example}

Rather than starting with $\bQ$, the other obvious example is to take 
$\bQ[C_6/C_3]$ in $\bQ[C_6/C_3] \leftmod$.

\begin{example}\label{ex:freeongroupring} We can calculate the following:
\[
\begin{array}{rcl}
F_3 (\bQ[C_6/C_3]) (C_6/C_6) &=& (\mathbb{Q}[(C_6/C_6)^{C_3}] \otimes \bQ[C_6/C_3])^{C_6/C_3}= \bQ \\
F_3 (\bQ[C_6/C_3]) (C_6/C_3) &=& (\mathbb{Q}[(C_6/C_3)^{C_3}] \otimes \bQ[C_6/C_3])^{C_6/C_3}= \bQ[C_6/C_3] \\
F_3 (\bQ[C_6/C_3]) (C_6/C_2) &=& (\mathbb{Q}[(C_6/C_2)^{C_3}] \otimes \bQ[C_6/C_3])^{C_6/C_3}= 0 \\
F_3 (\bQ[C_6/C_3]) (C_6/C_1) &=& (\mathbb{Q}[(C_6/C_1)^{C_3}] \otimes \bQ[C_6/C_3])^{C_6/C_3}= 0.
\end{array}
\]
The non-trivial restriction map is the inclusion of the fixed points. 
The induction map is the augmentation map (which
sends both $\underline{0}$ and $\underline{1}$ to $1 \in \bQ$). 
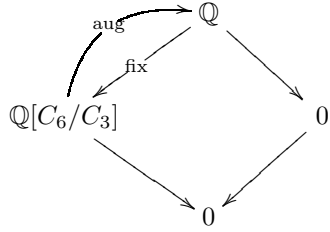
\[
\xymatrix{
& \bQ 
\ar[dl]|{\textrm{fix}} \ar[dr] 
\\
\bQ[C_6/C_3]  \ar[dr] \ar@/^1.5pc/[ur]|{\textrm{aug}}
&   &
\ \  0 \ \  \ar[dl] 
\\
& 0} 
\]
\end{example}

\begin{example} Similarly, we have the Lewis diagram of $F_1 (\bQ[C_6])$
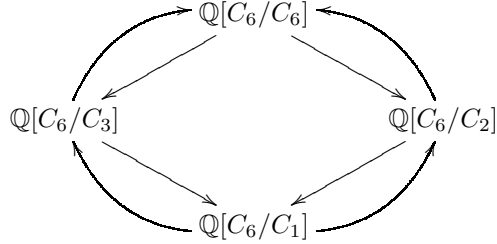
\[
\xymatrix{
& \bQ[C_6/C_6]
\ar[dl] \ar[dr] 
\\
\bQ[C_6/C_3]  \ar[dr] \ar@/^1.5pc/[ur]
&   &
\bQ[C_6/C_2] \ar[dl] \ar@/_1.5pc/[ul]
\\
& \bQ[C_6/C_1]
\ar@/^1.5pc/[ul] \ar@/_1.5pc/[ur]
}
\]
As with previous examples, the induction maps are given by the projections $C_6/K \to C_6/H$. 
The restriction maps are given by summing over the preimages of that projection. 

One may check this is isomorphic to the Mackey functor $\mackeyfunctor{C_6}(\bQ[C_6])$ 
of Example \ref{ex:fixedpointmackeyfunctors}. 
This functor is defined by $H \mapsto \bQ[C_6]^H$. The claimed 
isomorphism of Mackey functors is induced from the isomorphism
$\bQ[C_6/H] \lra \bQ[C_6]^H$, $gH \mapsto \sum_{h \in H} hg$.
\end{example}




\section{Comparison to equivariant spectra}

For $G$ a finite group, we have a classification of rational $G$-spectra in
terms of an algebraic model, see Theorem \ref{thm:finitealgmodel}. The algebraic model is built from
chain complexes of $\bQ[W_G H]$-modules for $H$
running over conjugacy classes of subgroups of $G$.

The most modern approach to the classification takes several steps
\begin{itemize}
\item idempotent splitting
\item restriction to normalisers
\item passing to Weyl groups (by taking fixed points)
\item algebraicisation
\end{itemize}
which we see are analogous to our classification of rational Mackey functors.
At the level of homotopy categories, one takes a spectrum $X$,
and then splits it into $e_H^G X$ for varying $H$. Then one
forgets to $N_G H$-spectra and takes $H$-fixed points.
Taking homology of the algebraicisation of a spectrum  $e_H^G X$
gives the homotopy groups of the spectrum.
This results in a graded $\bQ[W_G H]$-module
\[
\pi_* ( \big(i^* (e_H^G X)\big)^H ) = i^*(e_H^G) \pi_*^H(X)  =  \left( e_H^G ([-, X]^G_* \otimes \bQ)\right) (G/H).
\]
This is exactly the functor $U_H$ applied to the $e_H^G$-part of the Mackey functor
$[-, X]^G_* \otimes \bQ$.

A major difference between the method we use for Mackey functors and the
approach for rational $G$-spectra is that in the latter one proves that
the various model categories are Quillen equivalent at each stage,
rather than arguing via the composite functor.
This is partly due to adjunctions in the topological setting not being both
left and right adjoint and partly due to the difficulty of working with
complex composite functors in model categories.
For Mackey functors, we see that the adjunctions in
\[
\xymatrix@C+0.2cm{
e_H^G \mackey{G}
\ar@<0.7ex>[r]^-{\preextend} &
\ar@<0.7ex>[l]^-{\preforget }
R_{N_G H}^G (e_H^G)  \mackey{N_G H}
\ar@<0.7ex>[r]^-{\preinflate } &
\ar@<0.7ex>[l]^-{\otherfunctor }
\mackey{N_G H/H}
\ar@<0.7ex>[r]^-{(-)(W_G H/e)} &
\ar@<0.7ex>[l]^-{\mackeyfunctor{W_G H}}
\mathbb{Q}[W_G H] \leftmod
}
\]
are not equivalences. To resolve this, we can restrict
$\mackey{N_G H/H}$ and $R_{N_G H}^G (e_H^G)  \mackey{N_G H}$
to the full subcategories in the image of the functors
from $\mathbb{Q}[W_G H] \leftmod$.
We write  $\barmackey{N_G H/H}$
and
$R_{N_G H}^G (e_H^G)  \barmackey{N_G H}$
for these categories.
Our calculation of the counit of the $(F_H, U_H)$ adjunction shows that
\[
\id \to (\mackeyfunctor{W_G H} V )(W_G H/e)
\quad \textrm{ and } \quad
\id \to (\preinflate  \otherfunctor  \mackeyfunctor{W_G H} V )(W_G H/e)
\]
are isomorphisms. Hence the functors
$\otherfunctor $ and $\mackeyfunctor{W_G H}$ are full and faithful.
It follows that on these full subcategories, we have equivalences of categories.
\[
\xymatrix@C+0.2cm{
e_H^G \mackey{G}
\ar@<0.7ex>[r]^-{\preextend} &
\ar@<0.7ex>[l]^-{\preforget }
R_{N_G H}^G (e_H^G)  \barmackey{N_G H}
\ar@<0.7ex>[r]^-{\preinflate } &
\ar@<0.7ex>[l]^-{\otherfunctor }
\barmackey{N_G H/H}
\ar@<0.7ex>[r]^-{(-)(W_G H/e)} &
\ar@<0.7ex>[l]^-{\mackeyfunctor{W_G H}}
\mathbb{Q}[W_G H] \leftmod
}
\]
Passing to a full subcategory is the algebraic equivalent of localisation at an idempotent as used in the classification of rational $G$-spectra for finite $G$.

We will discuss the topological analogues of these results in more detail in Part \ref{part:2}.

\section{Monoidal properties}

We end this part with a discussion of the monoidal structure on Mackey functors.
Details can be found in Green \cite{green71} and Luca \cite{luca96}.
Given $G$-Mackey functors $M$ and $N$, we define
\[
T(H) = \bigoplus_{K \leqslant H} M(G/K) \otimes _\bQ N(G/K)
\quad \quad
(M \square N)(G/H)  =  T(H)/I(H)
\]
where $I(H)$ is the $\bQ$-submodule of $T(H)$ generated by
\begin{align*}
R_L^K (x) \otimes y' - x \otimes I_L^K(y') \quad \quad   &
\text{for } x \in M(G/K), \ y' \in N(G/L), \ L \leqslant K \leqslant H   \\
x' \otimes R_L^K (y)  - I_L^K(x') \otimes y  \quad \quad  &
\text{for } x' \in M(G/L), \ y \in N(G/K), \ L \leqslant K \leqslant H   \\
C_h(x) \otimes y - x \otimes C_{h}^{-1}(y) \quad \quad  &
\text{for } x \in M(G/K), \ y \in N(G/hKh^{-1}), \ L \leqslant K \leqslant H .
\end{align*}

\begin{thm}
For $M$ and $N$ $G$-Mackey functors, the construction
\[
(M \square N)(G/H)  =  T(H)/I(H)
\]
defines a Mackey functor when equipped with the conjugation, restriction and induction maps
described below.
We call this Mackey functor the \emph{box product} of $M$ and $N$.
\end{thm}

Conjugation is given by the diagonal action
\[
C_h (x \otimes y) = C_h (x) \otimes C_h(y).
\]
Induction from $H$ to $H'$ is given by the inclusion
\[
\Big( \bigoplus_{K \leqslant H} M(G/K) \otimes _\bQ N(G/K) \Big)
\lra
\Big( \bigoplus_{K \leqslant H'} M(G/K) \otimes _\bQ N(G/K) \Big)
\]
followed by taking quotients with respect to $I(H)$ and $I(H')$.
Restriction from $H'$ to $H$ is induced by the map
$T(H') \to T(H)/I(H)$ given by
\[
x \otimes y \longmapsto
\sum_{l\in[H \diagdown H' \diagup K]}
R_{H \cap l K l^{-1}}^{l K l^{-1}} C_l (x)
\otimes
R_{H \cap l K l^{-1}}^{l K l^{-1}} C_l (y)
\]
for $x \in M(K)$, $y \in N(K)$ and $K \leqslant H$.

One can also define the box product via a
convolution product (a left Kan extension over the product of $G$-sets),
using the definition of Mackey functors in terms of spans of $G$-sets
(the Burnside category), see \cite[Chapter 3]{luca96}. The unit for the box product is the
Burnside ring Mackey functor.

The equivalence of \cite[Chapter 3]{luca96} also implies that a (commutative) monoid for the
box product is a rational Mackey functor $M$, such that each $M(G/H)$
is a (commutative) $\bQ$-algebra, the conjugation and restriction maps
are maps of algebras and for $K \leqslant H$, the \emph{Frobenius relations} hold:
\[
x \cdot I_K^H (y) =   I_K^H (R_K^H(x) \cdot y)
\quad \quad
I_K^H (y) \cdot x =  I_K^H (y \cdot R_K^H(x))
\]
for $x \in M(G/H)$ and $y \in M(G/K)$. We call such a (commutative) monoid Mackey functor a (commutative) Green functor.

The category of $\mathbb{Q}[W_G H]$-modules
has a monoidal product, given by tensoring two modules over $\bQ$
and equipping the result with the diagonal $\bQ[W_GH]$-action.
We then see that $U_H$ sends (commutative) Green functors to
(commutative) monoids in $\mathbb{Q}[W_G H]$-modules.
In fact, we show that $U_H$ is a symmetric monoidal functor.

\begin{lem}\label{lem:box_prod}
Let $M$ and $N$ be Mackey functors.
Then
\[
e_H^G (M \square N )
\cong
(e_H^G M \square N )
\cong
(M \square e_H^G N)
\cong
(e_H^G M \square e_H^G N).
\]
Hence
\[
\left( e_H^G M \square e_H^G N \right) (G/H)  = e_H^G M(G/H) \otimes_\bQ e_H^G N(G/H).
\]
\end{lem}
\begin{proof}
The first statement is a calculation of the action of the Burnside ring on the box product.

For the second, the Mackey functor $e_H^G M$ is trivial on proper subgroups of $H$,
from which it follows that
\[
T(H) = e_H^G M(G/H) \otimes_\bQ e_H^G N(G/H)
\quad \textrm{ and } \quad
I(H) =0. \qedhere
\]
\end{proof}

The first statement of Lemma \ref{lem:box_prod} implies that the category $e^G_H \mackey{G}$ is monoidal with respect to $\square$ with the unit $e^G_H\burnsidering$.

\begin{cor}
For each $H \leqslant G$ the equivalence of categories
\[
\adjunction{U_H}{e^G_H \mackey{G}}{\bQ[W_G H] \leftmod}{F_H}
\]
is strong symmetric monoidal.

Moreover, the splitting result
\[
\mackey{G} \cong \prod_{H \leqslant_G G} e_H^G \mackey{G}
\]
is strong symmetric monoidal.
\end{cor}

The topological equivalent of this result is
Barnes, Greenlees and K\k{e}dziorek \cite{BGKeinfty}.
This gives a description of $E_\infty$-algebras in rational $G$-spectra
in terms of differential graded algebras in
\[
\prod_{(H) \leqslant G} \bQ[W_G H] \leftmod.
\]
The more complicated case of commutative ring $G$-spectra (or $E_\infty^G$-algebras, where $E_\infty^G$ is the operad governing the highest level of equivariant commutativity)
is considered in work of Wimmer \cite{wim19}.
The extra data here comes from multiplicative norm maps,
which are related to Tambara functors (commutative Green functors
with additional structure), see
Strickland \cite{stricktamb},
Mazur \cite{mazur13}
and Hill and Mazur \cite{HM19}.
The idempotent splitting result we use destroys the additional
structure of a Tambara functor, leaving only a commutative Green functor.
Hence, there is no immediate
extension of the above results to Tamabara functors.
The question of which idempotents and splittings persevere norms in the Burnside ring is 
answered fully in work of B\"{o}hme \cite{boehme19}.

\part{The structure of rational \texorpdfstring{$G$}{G}-spectra}\label{part:2}

For $G$ a compact Lie group, it is natural to study the homotopy theory of
$G$-spectra as Brown representability holds equivariantly, see \cite[Section XIII.3]{may96}. 
That is, $G$-equivariant cohomology theories are represented by $G$-spectra,
so the category of $G$-equivariant cohomology theories and
stable natural transformations between them, is equivalent to the homotopy category of $G$-spectra. 
Due to the complexity of the non-equivariant case, 
one cannot expect a complete analysis of either $G$-equivariant cohomology theories or $G$-spectra integrally.
However, if we restrict ourselves to $G$-equivariant cohomology theories with values in rational vector spaces,
the situation is greatly simplified,
whilst valuable geometric and group theoretic structures remain. 
For this reason, the programme of understanding $G$-equivariant cohomology theories begun by Greenlees restricts attention to rational $G$-equivariant cohomology
theories and rational $G$-spectra. 

In this part,  we discuss the methods and tools used to obtain algebraic models for rational equivariant spectra. 
Recall from the introduction that an algebraic model for rational $G$-spectra is a model category $d\mcA(G)$, 
that is Quillen equivalent to $G$-spectra. This category must consist of 
differential objects (and morphisms) in a graded abelian category $\mcA$.
To start our journey we begin by recalling some useful facts about $G$-spectra.

\section{Preliminaries on \texorpdfstring{$G$}{G}-spectra}

Let $G$ be a compact Lie group. We work with orthogonal $G$-spectra, 
see Mandell and May \cite{mm02} for more details. 
Unless otherwise stated, our categories of $G$-spectra will be indexed on a complete $G$-universe $\mcU$.

For $H$ a closed subgroup of $G$, one can define homotopy groups of an orthogonal $G$-spectrum $X$ with structure map 
$\sigma$ as 
\[
\pi_0^H(X)= \colimit_V [S^V, X(V)]^H
\]
where the maps in the colimit send a map $\alpha: S^V \lra X(V)$ to the composite
\[
S^W\cong S^V \wedge S^{V^\perp} \stackrel{\alpha \wedge \id}\lra X(V)\wedge S^{V^\perp}  \stackrel{\sigma}\lra X(V\oplus V^\perp) \cong X(W).
\]
Here $V$ runs through the $G$ representations in the universe $\mcU$. More generally, the integer graded homotopy groups of a $G$-spectrum $X$ are defined using shift and loop functors on spectra and the formula above.
A map $f$ of $G$-spectra is a weak equivalence, also called a \emph{stable equivalence}, in orthogonal $G$-spectra 
if and only if $\pi^H_p(f)$ is an isomorphism for all closed subgroups $H$ of $G$ and all integers $p$. The class of stable equivalences is part of a stable model structure on $G$-spectra, $G\endash\mathrm{Sp}^\cO$.

Orthogonal $G$-spectra with the stable model structure is a convenient model category 
for $G$-equivariant homotopy theory. 
In particular, the homotopy category is a symmetric monoidal triangulated category (see \cite[Appendix 2]{Margolis} for introduction to triangulated categories)
with unit the sphere spectrum $\mathbb{S}$, see Hovey, \cite[Section 7]{hov99}. 
Furthermore, the stable  equivalences can be detected by objects in the category in the following sense.  
For a closed subgroup $H$ in $G$, an orthogonal spectrum $X$ and integers $p \geq 0$ and $q>0$ 
\begin{equation}\label{generatorsG/H}
[\Sigma^pS^0\wedge G/H_+,X]^G \cong \pi^H_p(X) \quad \quad [F_qS^0\wedge G/H_+,X]^G \cong \pi^H_{-q}(X)
\end{equation}
where $[-,-]^G$ denotes morphisms in the homotopy category of $G\endash\mathrm{Sp}^\cO$ and $F_q(-)$ is the left adjoint to the evaluation functor at $\bR^q$: $Ev_{\bR^q}(X)=X(\bR^q)$. 
In particular, $F_q(S^0)$ models $\mathbb{S}^{-q}$, the $q$-fold desuspension of the sphere spectrum.
We can put this relation between the shifts of $G/H_+$ and the weak equivalences into the formalism
of \cite[Section  2]{ss03stabmodcat}.

\begin{defn}\label{compactobj}
Let $\cC$ be a triangulated category with infinite coproducts. A full triangulated subcategory of $\cC$ (with shift and triangles induced from $\cC$) is called \emph{localising} if it is closed under coproducts in $\cC$. A set $\cP$ of objects of $\cC$ is called a  \emph{set of generators} if the only localising subcategory of $\cC$ containing objects of $\cP$ is the whole of $\cC$. An object of a stable model category is called a \emph{generator} if it is so when considered as an object of the homotopy category.

An object $X$ in $\cC$ is  \emph{homotopically compact}\footnote{There are different names used in the literature - compact, small. We chose to use the name \emph{homotopically compact} here.}.
if for any family of objects $\{A_i\}_{i \in I}$ the canonical map
$$\bigoplus_{i\in I}[X,A_i]^{\cC} \lra [X, \coprod_{i \in I}A_i]^{\cC}$$ is an isomorphism in the homotopy category of $\cC$.
\end{defn}

The set of suspensions and desuspensions of $G/H_+$, where $H$ varies through all closed subgroups of $G$, is a set of 
homotopically compact generators in the stable model category $G\endash\mathrm{Sp}^\cO$. Those objects are compact since homotopy groups commute with coproducts and it is clear from \cite[Lemma 2.2.1]{ss03stabmodcat} and Equation (\ref{generatorsG/H}) that this is a set of generators for $G\endash\mathrm{Sp}^\cO$.

There is an easy-to-check condition for a Quillen adjunction between stable model categories with sets of homotopically compact generators to be a Quillen equivalence. It is used often in the setting of algebraic models. 
Also notice that the derived functors of Quillen equivalences preserve homotopically compact objects.

\begin{lem}\label{qequiv_generators}Suppose $F: \cC \rightleftarrows \mcD : U$ is a Quillen pair between stable model categories with sets of homotopically compact generators, such that the right derived functor $RU$ preserves coproducts (or equivalently, such that the left derived functor sends homotopically compact generators to homotopically compact objects). 

If the derived unit and counit are weak equivalences for the respective sets of generators, 
then  $(F,U)$ is a Quillen equivalence.
\end{lem}
\begin{proof} The result depends upon the fact that the homotopy category of a stable model category is a triangulated category. First notice that since the derived functor $RU$ preserves coproducts, the derived unit and counit are triangulated natural transformations. If the derived unit condition is an isomorphism for a set of objects $\mathcal{K}$ then they are also satisfied for every object in the localising subcategory for $\mathcal{K}$. Since we assume that $\mathcal{K}$ consists of homotopically compact generators, the localising subcategory for $\mathcal{K}$ is the whole category
and the derived unit is an isomorphism. The same argument applies to the counit and the result follows.
\end{proof}

To construct a model category of \emph{rational} $G$-spectra will we need to 
introduce the language of Bousfield localisations, see Section \ref{sec:locals}.
Since we will often localise the model category of rational $G$-spectra at idempotents
of the rational Burnside ring, we first look at this ring.

\section{Idempotents of the rational Burnside ring}\label{sec:idempotents}

For $G$ a compact Lie group, the Burnside ring $\A(G)$ was defined by tom Dieck in \cite{tomDieck_Burnside_ring_CLG1}
in terms of $G$-manifolds. 
For a survey on the subject see, for example, Fausk \cite{Fausk_Burnside}.
When working rationally, several descriptions of this ring exist. We give these descriptions
and use them to understand the idempotents of the rational Burnside ring. 
These idempotents are fundamental to the construction of the algebraic model
and the calculations therein.

\subsection{Two ways of understanding rational Burnside ring}\label{sec:two_ways_idempotents}
Recall that for $H$ a subgroup of $G$,  $N_GH=\{g \in G\ |\ gH=Hg\}$ is the normaliser of $H$ in $G$. 
We write $W=W_GH=N_GH/H$ for the Weyl group of $H$ in $G$.

Let $\mcF(G)$  be the set of closed subgroups of $G$ with finite index in their normalizer. 
That is,  all closed  $H\leqslant G$ such that $N_GH/H$ is finite. 
We give this set the topology induced by the Hausdorff metric, see \cite[Section V.2]{lms86}. 

The isomorphism of tom Dieck (see Lemma \ref{lem:IsoBurnsideFinite}) can be extended to 
compact Lie groups by \cite[ Propositions 5.6.4 and 5.9.13]{tomdieck1} giving an isomorphism of rings 
\begin{equation}\label{eq:tomDieck_iso}
\A(G)\otimes \bQ \cong C(\mcF(G)/G,\bQ),
\end{equation}
where $C(\mcF(G)/G,\bQ)$ denotes the ring of continuous functions 
on the orbit space $\mcF(G)/G$ with values in discrete space $\bQ$.
From now on, we will use notation $\A_\bQ(G)$ for  $\A(G)\otimes \bQ$. 
This isomorphism generalises that of Lemma \ref{lem:IsoBurnsideFinite}.
Notice that if $G$ is a finite group, 
then $\mathrm{Sub}(G)=\mcF(G)$, where $\mathrm{Sub}(G)$ is the set of all subgroups of $G$.

From the ring isomorphism above, it is clear that idempotents of the rational Burnside ring of $G$ correspond to the characteristic functions of open and closed subspaces of the orbit space  $\mcF(G)/G$. 
 By the characteristic function of a set $V$, we mean the function that takes the value $1$ for every point in $V$ and $0$ otherwise. Equivalently, an idempotent corresponds to an open and closed $G$-invariant subspace of $\mcF(G)$.
We write $e_{V}$, for the idempotent corresponding to an open and closed subset $V$ of $\mcF(G)/G$.

Every inclusion $i: H \lra G$ induces a ring homomorphism $i^* \colon \A_\bQ(G) \lra \A_\bQ(H)$. 
In general, it is difficult to 
explicitly describe the image of a given idempotent in terms of open and closed sets under 
\[
i^*: C(\mcF(G)/G,\bQ) \lra C(\mcF(H)/H,\bQ).
\] 
Even before taking conjugacy classes into account, notice that a subgroup $K \leqslant H$ with finite index in the normaliser $N_HK$ does not have to have a finite index in the normaliser $N_GK$.  Thus the map $i^*: C(\mcF(G)/G,\bQ) \lra C(\mcF(H)/H,\bQ)$ is not always induced by a map from $\mcF(H)$ to $\mcF(G)$. The exception is of course, when $G,H$ are finite groups, as we discussed in Remark \ref{rmk:restrictidempotents}.

A better approach to investigate the action of $i^*$ on idempotents is to view idempotents as 
corresponding to certain subspaces of the space of \emph{all closed subgroups} of $G$ as follows. 
We put a topology on the set of all closed subgroups of $G$, $\mathrm{Sub}(G)$.
This topology is called the $\mathrm{f}$-topology in Greenlees \cite[Section 8]{greratmack}.

For a closed subgroup $H\leqslant G$ and $\varepsilon >0$ we define a ball
$$
O(H,\varepsilon)=\{K\in \mcF (H) \mid d(H,K)<\varepsilon \}
$$
in $\mathrm{Sub}(G)$, where the distance above is measured with respect to the Hausdorff metric. Thus, subgroups close to $H$ that have infinite Weyl groups are ignored, for example if $H=SO(2)$ is a torus then $O(SO(2),\varepsilon)$ is a singleton. Given also a neighbourhood $A$ of the identity in $G$ consider $$O(H,\varepsilon, A)=\cup_{a\in A}O(H,\varepsilon)^a,$$
where $O(H,\varepsilon)^a$ is the set of $a$-conjugates of elements of $O(H,\varepsilon)$. 

\begin{defn}
For $G$ a compact Lie group, the 
$\mathrm{f}$-topology on $\mathrm{Sub}(G)$ is generated by 
the sets $O(H,\varepsilon, A)$ as $H, \varepsilon, A$ vary.
We write $\mathrm{Sub_f}(G)$ for this topological space. 
\end{defn}

We say that subgroups $K \leqslant H$ of $G$ are \emph{cotoral} if $H/K$ is a torus. 
We write $K \sim H$ for the equivalence relation generated by the cotoral pairs. 
An idempotent in a rational Burnside ring $\A_\bQ(G)$ corresponds to an open and closed, $G$-invariant subspace of $\mathrm{Sub_f}(G)$ that is a union of $\sim$-equivalence classes,  see Section \ref{sec:examples_idempotents_and_cotoral} for some examples.

Let $V$ be an open and closed $G$-invariant set in $\mathrm{Sub_f}(G)$ that is a union of $\sim$-equivalence classes.
Let $i^* V$ be the preimage of $V$ under $i^* \colon \mathrm{Sub_f}(H) \lra \mathrm{Sub_f}(G)$. 
We then let $\overline{i^*V}$ be the smallest $G$-invariant open and closed set of $\mathrm{Sub_f}(H)$, 
that is the union of $\sim$-equivalence classes containing $i^*V$. 
Using the techniques of Greenlees \cite[Section 8]{greratmack}
one can show the following.

\begin{lem}
Let $i \co H \to G$ be an inclusion of a closed subgroup.
Let $e_V$ an idempotent of $\A_\bQ(G)$ corresponding to 
$V$, an open and closed $G$-invariant set in $\mathrm{Sub_f}(G)$ that is a union of $\sim$-equivalence classes.
Then $i^*(e_V) =e_{\overline{i^*V}}$.
\end{lem}

\begin{rem} 
As it is more common in the literature, an idempotent $e_V$ will come from a 
open and closed subset $V$ of $\mcF(G)/G$ unless otherwise stated. 
\end{rem}

\subsection{Special idempotents}

There are two situations of particular interest to us. In these cases we have an idempotent in the rational Burnside ring and we can provide an algebraic model for the piece of homotopy theory of rational $G$-spectra that this idempotent governs. 

The first situation is where there is an idempotent that remembers only one, special subgroup. The second author called such a subgroup \emph{exceptional} in \cite{KedziorekExceptional}. The second situation is where the idempotent corresponds to the maximal torus $\T$ in $G$ and all its subgroups. This is called the \emph{toral part} of rational $G$-spectra in \cite{toralBGK}.

When we look at idempotents defined by subsets of $\mcF(G)/G$ the above two cases look identical at first glance: both idempotents are indexed by one subgroup. However in the case of a torus, there are subgroups 
of the torus that are ``hidden'' in the torus idempotent. 
This is visible when one uses the space $\mathrm{Sub_f}(G)$ to describe the idempotent. 
The subgroups that are cotoral in $\T$ are responsible for making the algebraic model for that part substantially more difficult than in the case of an exceptional subgroup.

We will start our analysis with the case of an exceptional subgroup of $G$.

\begin{defn}\label{defn:exceptional}
Suppose $G$ is a compact Lie group. We say that a closed subgroup $H \leqslant G$ is \emph{exceptional}\footnote{the name was motivated by the \emph{exceptional} behaviour of the algebraic model over such a subgroup.} in $G$ if $W_GH$ is finite, there exists an idempotent $e_H^G$ in the rational Burnside ring of $G$ corresponding to the conjugacy class of $H$ in $G$ 
(via tom Dieck's isomorphism, Equation \ref{eq:tomDieck_iso}) and $H$ has no cotoral subgroups. 
\end{defn}

If $H$ is exceptional in $G$, then $\{K \mid K\in (H)_G\}$ is an open and closed $G$-invariant subspace of $\mathrm{Sub_f}(G)$, which already is a union of $\sim$-equivalence classes, since $H$ does not contain any cotoral subgroup and $W_GH$ is finite. The other implication also holds; if there is an idempotent corresponding to $\{K|K\in (H)_G\}$ in $\mathrm{Sub_f}(G)$, then $H$ is an exceptional subgroup of $G$. Thus we could rephrase the definition in terms of the space $\mathrm{Sub_f}(G)$, but we decided to use the more familiar $\mcF (G)/G$ with the topology given by the Hausdorff metric.

Any subgroup of a finite group $G$ is exceptional. In $O(2)$ only finite dihedral subgroups are exceptional; in particular none of the finite cyclic subgroups are exceptional (since finite cyclic subgroups do not have idempotents in the rational Burnside ring of $O(2)$). The maximal torus $SO(2)$ in $O(2)$ has an idempotent in the rational Burnside ring of $O(2)$, however it is not an exceptional subgroup, since it contains cotoral subgroups, for example the trivial one. In $SO(3)$ all finite dihedral subgroups are exceptional except for $D_2$, which is conjugate to $C_2$ and therefore is a cotoral subgroup of a torus. There are four more conjugacy classes of exceptional subgroups: $A_4,\Sigma_4, A_5$ and $SO(3)$, where $A_4$ denotes rotations of a tetrahedron, $\Sigma_4$ denotes rotations of a cube and $A_5$ denotes rotations of a dodecahedron, 
see \cite[Section 2]{KedziorekSO(3)}.

If a trivial subgroup is exceptional in $G$, then $G$ has to be finite. This holds as the normaliser of a trivial subgroup is the whole $G$, $W_G \{1\}=G$ and the condition that the Weyl group is finite implies that $G$ is a finite group.

Given an exceptional subgroup $H$, we may use the corresponding idempotent in the rational Burnside ring to split 
(see Section \ref{sec:locals})
the category of rational  $G$-spectra into the part over an exceptional subgroup $H$ and its complement. 
\cite{KedziorekExceptional} presents the model for rational $G$-spectra over an exceptional subgroup $H$.

The exceptional subgroups of a group $G$ can be divided into two sets, 
according to how their idempotent behaves once restricted to the normaliser
of the exceptional subgroup. 
We closely follow \cite{KedziorekExceptional} in analysis of these different behaviours.

\begin{defn}
\label{goodSub}
Suppose $H\leqslant A$ are closed subgroups of $G$ such that $H$ is exceptional in $G$. Suppose further that $i:A \lra G$ is an inclusion. We say that $H$ is $A$-{\emph{good}} in $G$ if $i^*(e_H^G)= e_H^A$ and  $A$-{\emph {bad}} in $G$ if it is not $A$-good, i.e. $i^*(e_H^G)\neq e_H^A$.
\end{defn}

Notice that the above definition is all about subgroups conjugate to $H$ in $A$ and in $G$ and their relation to each other. If $L\leqslant A$ is such that $L$ is conjugate to $H$ in $A$, then it is also true that $L$ is conjugate to $H$ in $G$. Thus if $H$ is $A$-bad in $G$ it just means that there exists $L'\leqslant A$ such that $(L')_G=(H)_G$ and $(L')_{A}\neq(L)_A$.
An exceptional subgroup $H$ in a compact Lie group $G$ is always $H$-good in $G$.

\begin{lem}\label{goodAndBadSub}\cite[Lemma 2.3]{KedziorekExceptional} For the exceptional subgroups in $G=SO(3)$, 
we have the following relation between $H$ and its normaliser $N_GH$:
\begin{enumerate}
\item $A_5$ is $A_5$-good in $SO(3)$.
\item $\Sigma_4$ is $\Sigma_4$-good in $SO(3)$.
\item $A_4$ is $\Sigma_4$-good in $SO(3)$.
\item $D_4$ is $\Sigma_4$-bad in $SO(3)$.
\end{enumerate}
\end{lem}
\begin{proof} We only need to prove Part (3) and (4), since any exceptional subgroup $H$ in a compact Lie group $G$ is $H$-good in $G$. Part (3) follows from the fact that there is one conjugacy class of $A_4$ in $\Sigma_4$, as there is just one subgroup of index 2 in $\Sigma_4$. Part (4) follows from the observation that there are two subgroups of order 4 in $D_8$ (so also in $\Sigma_4$) and they are conjugate by an element $g\in D_{16}$, which is the generating rotation by 45 degrees (thus $g \notin D_8$ and thus $g\notin \Sigma_4$).
\end{proof}

\begin{rem} Notice that we can generalise Definition \ref{goodSub} to non-exceptional subgroups using the equivalent description in terms of conjugacy classes of $H$ in $A$ and in $G$. In that case, if $G=SO(3)$, $A=O(2)$ and $H=C_2 \leqslant A$, then $H$ is $A$-bad in $G$, which follows from the fact that $D_2 \leqslant A$ is $G$-conjugate to $H$, but not $A$-conjugate. This \emph{bad} behaviour of $C_2$ in $SO(3)$ is visible in the adjunctions used to obtain the algebraic model for toral part of rational $SO(3)$-spectra in \cite{KedziorekSO(3)}, which we recall in Proposition \ref{notQErest_ind}.
\end{rem}

Finishing the discussion about idempotents of rational Burnside ring, we note that there is always an idempotent corresponding to the maximal torus $\T$ in $G$ and all its subgroups. This fact was used in \cite{toralBGK} to obtain an algebraic model for rational toral $G$-spectra, thus the ones that have geometric isotropy contained in the set of subgroups of the maximal torus.

\subsection{Examples}\label{sec:examples_idempotents_and_cotoral}

\subsubsection{Closed subgroups of \texorpdfstring{$SO(2)$}{SO(2)}}

Recall that $SO(2)$ is the group of rotations of $\RR^2$. The closed subgroups of $SO(2)$ are 
the finite cyclic groups $C_n$. Each $C_n$ is cotoral in $SO(2)$, 
that is, it is normal in $SO(2)$ and $SO(2)/C_n \cong SO(2)$.
The only subgroup of $SO(2)$ with finite index in its normaliser is $SO(2)$ itself.
Hence, the space $\mcF(SO(2))/SO(2)$ is a single point and 
the rational Burnside ring of $SO(2)$ is $\bQ$.
Similar arguments show that $\A_\bQ(\T)= \bQ$ for $\T$ a torus of any rank.

\subsubsection{Closed subgroups of \texorpdfstring{$O(2)$}{O(2)}}\label{sec:subgroupsO2}

Recall that $O(2)$ is the group of rotations and reflections of $\RR^2$.
The closed subgroups are the finite cyclic groups, $T=SO(2)$, $O(2)$
and finite dihedral groups. For fixed $n$, the finite dihedral groups of order $2n$ are all
conjugate. We Write $D_{2n}$ for this conjugacy class.  
The space $\mcF(O(2))/O(2)$ consists of two parts, which we call the toral part and the dihedral part. 
The toral part $\tcT$, is just one point $T$ corresponding to the maximal torus and all its subgroups.
The dihedral part $\tcD$, is the set of all dihedral subgroups together with their limit point $O(2)$. 
Thus, we have idempotents $e_{\tcT}$ and $e_{\tcD}$ 
in the rational Burnside ring of $O(2)$ that sum to the identity. 

The toral idempotents for $O(2)$ and $SO(3)$ will behave very differently when we discuss the 
interactions between localisations and change of group functors
in Section \ref{sec:change_localisation}.
To help the notation for this comparison, we use a tilde to 
denote the dihedral and toral parts of $\mcF(O(2))/O(2)$
and no tilde for $SO(3)$.

\begin{center}
\begin{tikzpicture}[dot/.style={circle, inner sep=1pt,fill,label={#1},name=#1,color=black}]
\usetikzlibrary{calc}
\coordinate[label=$\mathrm{Space}\ \mcF(O(2))/O(2)$] (O) at (3,6);
\coordinate[label=$\mathrm{Part\ (Subspace)}$] (O) at (-3,6);

\coordinate [label=$\tcT$]  (FF) at (-3,5.5);
\coordinate[label=above:$T$] (M) at (0,5.5);

\coordinate [label=$\tcD$]  (FG) at (-3,4);
\coordinate [label= above:$D_2$]  (F) at (0,4);
\coordinate [label= above:$D_4$]  (G) at (3,4);
\coordinate [label= above:$D_{6}$]  (H) at (5,4);
\coordinate [label= above:$D_{8}$]  (I) at (6,4);
\coordinate [label= above:$D_{10}$]  (L) at (6.6,4);
\coordinate[label=$...$] (J) at (7.3,3.86);
\coordinate[label=above:$O(2)$] (K) at (8,4);

\fill[black] (F) circle (2pt);
\fill[black] (G) circle (2pt);
\fill[black] (H) circle (2pt);
\fill[black] (I) circle (2pt);
\fill[black] (K) circle (2pt);
\fill[black] (L) circle (2pt);
\fill[black] (M) circle (2pt);

\end{tikzpicture}
\end{center}

\subsubsection{Closed subgroups of \texorpdfstring{$SO(3)$}{SO(3)}}\label{sec:subgroupsSO3}

Recall that $SO(3)$ is a group of rotations of $\RR^3$. We choose a maximal torus $T$ in $SO(3)$ with rotation axis the $z$-axis. We divide the closed subgroups of $G$ into three types: \emph{toral} $\cT$, \emph{dihedral} $\mcD$ 
and \emph{exceptional} $\mcE$. This division is motivated by 
our preferred splitting of the category of rational $SO(3)$-spectra. 
The toral part consist of all tori in $SO(3)$ and all cyclic subgroups of these tori. 
Note that for any natural number $n$ there is one conjugacy class of subgroups from the toral part of order $n$ in $SO(3)$.

The dihedral part consists of all dihedral subgroups $D_{2n}$ (dihedral subgroups of order $2n$) of $SO(3)$ where $n$ is greater than 2, together with all subgroups isomorphic to $O(2)$. 
Note that $O(2)$ is the normaliser for itself in $SO(3)$. 
Moreover, there is only one conjugacy class of a dihedral subgroup $D_{2n}$ for each $n$ greater than $2$.
The normaliser of $D_{2n}$ in $SO(3)$ is $D_{4n}$ for $n>2$. 

We deliberately exclude the conjugacy classes of $D_2$ and $D_4$ from the dihedral part. 
Conjugates of $D_2$ are excluded from the dihedral part, as $D_2$ is conjugate to $C_2$ in $SO(3)$ and that subgroup is already taken into account in the toral part. Conjugates of $D_4$ are excluded from the dihedral part since its normaliser in $SO(3)$ is $\Sigma_4$ (symmetries of a cube), thus its Weyl group $\Sigma_4/D_4$ is of order $6$, whereas all other finite dihedral subgroups $D_{2n}, n>2$ have Weyl groups of order $2$. For simplicity we decided to treat $D_4$ separately.

There are five conjugacy classes of subgroups which we call exceptional, namely $SO(3)$ itself, the rotation group of a cube $\Sigma_4$,  the rotation group of a tetrahedron $A_4$, the rotation group of a dodecahedron $A_5$  and $D_4$, the dihedral group of order 4. Normalisers of these exceptional subgroups are as follows: $\Sigma_4$ is equal to its normaliser, $A_5$ is equal to its normaliser and the normaliser of $A_4$ is $\Sigma_4$, as is the normaliser of $D_4$.

Consider the space $\mcF(SO(3))/SO(3)$ of conjugacy classes of subgroups of $SO(3)$ with finite index in their normalisers. Recall that the topology on this space is induced by the Hausdorff metric. The division into these parts is an indication of idempotents of the rational Burnside ring for $SO(3)$ that are chosen to obtain an algebraic model for rational $SO(3)$-spectra.

\begin{center}
\begin{tikzpicture}[dot/.style={circle, inner sep=1pt,fill,label={#1},name=#1,color=black}]
\usetikzlibrary{calc}
\coordinate[label=$\mathrm{Space}\ \mcF(SO(3))/SO(3)$] (O) at (3,8);
\coordinate[label=$\mathrm{Part\ (Subspace)}$] (O) at (-3,8);
\coordinate[label=$\mcE$] (F) at (-3,7);
\coordinate [label= above:SO(3)]  (A) at (0,7);
\coordinate [label= above:$\Sigma_4$] (B) at (2,7) ;
\coordinate [label= above:$A_4$] (C) at (4,7) ;
\coordinate [label= above:$A_5$] (D) at (6,7) ;
\coordinate [label= above:$D_4$] (E) at (8,7) ;

\fill[black] (A) circle (2pt);
\fill[black] (B) circle (2pt);
\fill[black] (C) circle (2pt);
\fill[black] (D) circle (2pt);
\fill[black] (E) circle (2pt);

\coordinate [label=$\cT$]  (FF) at (-3,5.5);
\coordinate[label=above:$T$] (M) at (0,5.5);

\coordinate [label=$\mcD$]  (FG) at (-3,4);
\coordinate [label= above:$D_{6}$]  (H) at (5,4);
\coordinate [label= above:$D_{8}$]  (I) at (6,4);
\coordinate [label= above:$D_{10}$]  (L) at (6.6,4);
\coordinate[label=$...$] (J) at (7.3,3.86);
\coordinate[label=above:$O(2)$] (K) at (8,4);

\fill[black] (H) circle (2pt);
\fill[black] (I) circle (2pt);
\fill[black] (K) circle (2pt);
\fill[black] (L) circle (2pt);
\fill[black] (M) circle (2pt);

\end{tikzpicture}
\end{center}

\ \\

The topology on $\mcE$ is discrete, $\cT$ consists of one point $T$ and $\mcD$ forms a sequence of points converging to $O(2)$.

Note the difference between the dihedral parts for $O(2)$ and $SO(3)$: 
the conjugacy class of $D_2$ and $D_4$.
At a first glance, the toral part for $SO(3)$ looks the same as the toral part for $O(2)$. However, for $SO(3)$ it contains information about $D_2\leqslant O(2)$ (since $D_2$ is conjugate to $C_2$ in $SO(3)$), whereas for $O(2)$ it does not. These differences will become significant when we look at the interactions between localisations at idempotents and change of groups functors in Section \ref{sec:change_localisation}.

We use the following idempotents in the rational Burnside ring of $SO(3)$: $e_{\cT}$ corresponding to the characteristic function of the toral part $\cT$, $e_{\mcD}$ corresponding to the characteristic function of the dihedral part $\mcD$ and $e_{\mcE}$ corresponding to the characteristic function of the exceptional part $\mcE$. Since $\mcE$ is a disjoint union of five points, it is in fact a sum of five idempotents, 
one for every (conjugacy class of a) subgroup in the exceptional part: $e_{SO(3)}$, $e_{\Sigma_4}$, $e_{A_4}$, $e_{A_5}$ and $e_{D_4}$. We use a simplified notation $e_H$ to mean $e_H^{SO(3)}$ here.

\begin{rem} All finite dihedral subgroups in $SO(3)$ are exceptional, hence each has an idempotent
corresponding to it. 
However, as there are countably many conjugacy classes of dihedral subgroups, 
we cannot write $e_{\mcD}$ as the sum
of all these idempotents. 
Similarly, the characteristic function of the point $O(2)$ is not a continuous map to $\bQ$, 
hence it does not correspond to an idempotent.
\end{rem}

\section{Left and right Bousfield localisations and splittings}\label{sec:locals}

There are two well-understood ways of making a homotopy category of a given model category \emph{smaller}. Both ways boil down to adding weak equivalences in a tractable way. The first one keeps the cofibrations the same and is called a left Bousfield localisation (the particular version we use is also called a \emph{homological localisation}). 
The second one keeps the fibrations the same and is called the right Bousfield localisation (or cellularisation).

\subsection{Left Bousfield localisation} 
The general theory of left Bousfield localisations is given in Hirschhorn \cite{hir03}. 
For homological localisation we use 
the following result, which is \cite[Chapter IV, Theorem 6.3]{mm02}.

\begin{thm}Suppose $E$ is a cofibrant object in $G\endash\mathrm{Sp}^\cO$ or a cofibrant based $G$-space. There is a new model structure called the \emph{$E$-local model structure} on $G\endash\mathrm{Sp}^\cO$, denoted 
$L_E(G\endash\mathrm{Sp}^\cO)$, defined as follows. 
A map $f \colon X \lra Y$ is
\begin{itemize}
\item a weak equivalence if it is an $E$-equivalence, that is, $\id_E\wedge f \colon E\wedge X \lra E\wedge Y$ is a stable equivalence,
\item a cofibration if it is a cofibration with respect to the stable model structure,
\item a fibration if it has the right lifting property with respect to all trivial cofibrations.
\end{itemize}
The $E$-fibrant objects $Z$ are the fibrant $G$-spectra that are $E$-local, that is, the map 
\[
[f, Z]^G \colon [Y,Z]^G \lra [X,Z]^G
\] 
is an isomorphism for all $E$-equivalences $f$.
For $X$ a $G$-spectrum, $E$-fibrant approximation gives Bousfield localisation $\lambda \colon X\lra L_EX$ of $X$ at $E$.
\end{thm}

We will refer to the above model structure as the left Bousfield localisation of the category of  $G$-spectra at $E$.
This model category is proper, stable, symmetric monoidal and cofibrantly generated.  
An $E$-equivalence between $E$-local objects is a weak equivalence by \cite[Theorems 3.2.13 and 3.2.14]{hir03}.

As previously mentioned, the first simplification of the category of  $G$-spectra is rationalisation.
This means localisation at the Moore spectrum for $\bQ$,  $\mathbb{S}_\bQ$. 
For details see \cite[Definition 5.1]{barnessplitting}. 
This spectrum has the property that $\pi_*(X \wedge \mathbb{S}_\bQ)=\pi_*(X)\otimes \bQ$.
We refer to this model category as the model category of rational  $G$-spectra.

The self-maps of the rational sphere spectrum in the homotopy category of $G$-spectra are given by the rational Burnside ring
\[
\A_\bQ(G) 
\cong 
[\mathbb{S},\mathbb{S}]^{G\endash\mathrm{Sp}^\cO} \otimes \bQ
\cong 
[\mathbb{S},\mathbb{S}]^{L_{\mathbb{S}_\bQ}G\endash\mathrm{Sp}^\cO}
\cong
[\mathbb{S}_\bQ,\mathbb{S}_\bQ]^{G\endash\mathrm{Sp}^\cO}.
\]
It follows that $e \in \A_\bQ(G)$
can be represented by a map $e \colon \mathbb{S}_\bQ \lra \mathbb{S}_\bQ$.  
We define $e\mathbb{S}_\bQ$ to be the homotopy colimit (a mapping telescope) of the diagram
\[
\xymatrix{\mathbb{S}_\bQ \ar[r]^{e} & \mathbb{S}_\bQ\ar[r]^{e} & \mathbb{S}_\bQ\ar[r]^{e} &...\ .}
\]
We ask for this spectrum to be cofibrant either by choosing a good construction of homotopy colimit, or by cofibrantly replacing the result in the stable model structure for $G$-spectra. 
We thus have model structures 
$L_{e\mathbb{S}_\bQ}(G\endash\mathrm{Sp})$ and  $L_{(1-e)\mathbb{S}_\bQ}(G\endash\mathrm{Sp})$.
Fibrant replacement in $L_{e\mathbb{S}_\bQ}(G\endash\mathrm{Sp}^\cO)$ 
is given by taking the fibrant replacement of $X \smashprod e\mathbb{S}_\bQ$.
Since this commutes with taking infinite coproducts, the localisation is 
smashing in the sense of Ravenel \cite{Ravenel_Localization} and Hovey et al.\ \cite{hps97}).
In particular, this localisation preserves homotopically compact generators.

We know from Section \ref{sec:idempotents} that $e$ corresponds to an open and closed, $G$-invariant subspace of $\mathrm{Sub_f}(G)$ that is a union of $\sim$-equivalence classes, call it $V_e$.
By considering the geometric fixed point functors $\Phi^H$, for all $H \leqslant G$ 
(see \cite[Section V.4]{mm02}), we can see that the homotopy category of 
$L_{e\mathbb{S}_\bQ}(G\endash\mathrm{Sp}^\cO)$  
is the homotopy category of rational  $G$-spectra $X$ with geometric isotropy 
\[
\mcG\mcI(X)=\{H\leqslant G \mid \Phi^H(X)\not\simeq *\},
\]
concentrated over the subgroups $H$ that are in $V_e$.

\subsection{Splitting}
A common step in the classification of rational $G$-spectra is to split the 
category using idempotents of the rational Burnside ring. 
Work of the first author \cite{barnessplitting}
allows us to perform a compatible splitting at the level of model categories.

\begin{thm}\label{thm:splittingofexep}\cite[Theorem 4.4]{barnessplitting} Let $e$ be an idempotent in the rational Burnside ring $\A_\bQ(G)$. There is a strong symmetric monoidal Quillen equivalence:
\[
\xymatrix{
\triangle\ :\ L_{\mathbb{S}_\bQ}(G \endash\mathrm{Sp}^\cO)\ \ar@<+1ex>[r] & \ L_{e \mathbb{S}_\bQ}(G\endash\mathrm{Sp}^\cO)\times L_{(1-e) \mathbb{S}_\bQ}(G\endash\mathrm{Sp}^\cO)\ :\Pi \ar@<+0.5ex>[l].
}
\]
The left adjoint is a diagonal functor, the right adjoint is a product and the product category on the right is considered with the objectwise model structure (a map $(f_1,f_2)$ is a weak equivalence, a fibration or a cofibration if both factors $f_i$ are so).
\end{thm}

One can also look at splittings non-rationally, as in B{\"{o}}hme \cite{boehme19}.

\subsection{Cellularisation}\label{sec:cellularisation}

A cellularisation of a model category
is a right Bousfield localisation at a set of objects.
Such a localisation exists by \cite[Theorem 5.1.1]{hir03}
whenever the model category is right proper and cellular.
When we are in a stable context, the results of \cite[Section 5]{brstable}
can be used, which allows us to relax the cellularity condition. 

The most common use of   cellularisation in the context of algebraic models
is the Cellularisation Principle, which we recall in Theorem \ref{thm:cellprin}.

\begin{defn}
Let $\cC$ be a stable model category and $K$ a stable set of objects of $\cC$, i.e. a set such that a class of $K$-cellular objects of $\cC$ is closed under desuspension (Note that this class is always closed under suspension.) We call $K$ a set of \emph{cells}.
We say that a map $f \colon A \longrightarrow B$ of $\cC$ is a \emph{$K$-cellular equivalence} if
the induced map
\[
[k,f]^\cC_* \colon [k,A]^\cC_* \longrightarrow [k,B]^\cC_*
\]
is an isomorphism of graded abelian groups for each $k \in K$. An object $Z \in \cC$ is said to be
\emph{$K$-cellular} if
\[
[Z,f]^\cC_* \colon [Z,A]^\cC_* \longrightarrow [Z,B]^\cC_*
\]
is an isomorphism of graded abelian groups for any $K$-cellular equivalence $f$.
\end{defn}

The following is Hirschhorn \cite[Theorem 5.1.1]{hir03}.
\begin{thm}
For $K$ a set of objects in a right proper, cellular model category $\cC$, the 
\emph{right Bousfield localisation} or \emph{cellularisation} of $\cC$ with respect to
$K$ is the (right proper) model structure $K \cell \cC$ on $\cC$ defined as follows.
\begin{itemize}
\item The weak equivalences are $K$-cellular equivalences,
\item the fibrations of $K \cell \cC$ are the fibrations of $\cC$,
\item the cofibrations of $K \cell \cC$ are defined via left lifting property.
\end{itemize}
The cofibrant objects of $K \cell \cC$
are called $K$-cofibrant and are precisely the
$K$-cellular and cofibrant objects of $\cC$.
\end{thm}

When $\cC$ is stable and $K$ is a stable set of cofibrant objects, then 
the cellularisation of a proper, cellular stable model category 
is proper, cellular and stable by Barnes and Roitzheim 
\cite[Theorem 5.9]{brstable}.

We can further ask the cells $K$ to be homotopically compact objects. 
By \cite[Section 9]{brstable} the homotopy category $K \cell \cC$
is the full triangulated subcategory of the homotopy category of $\cC$
generated by $K$. In particular, $K$ is a set of homotopically compact 
generators for $K \cell \cC$. These ideas lead to the following theorem.
For examples of its use, see Section \ref{sec:alternativetosplitting}, Theorem \ref{thm:so3so2adjuncation} or Theorem \ref{thm:generalToral}.

\begin{thm}[The Cellularisation Principle]\label{thm:cellprin}
Let $M$ and $N$ be right proper,
stable, cellular model categories with $(F,U)$ a Quillen adjunction between $M$ and $N$. 
Let $\mcQ$ be a cofibrant replacement functor in $M$ and $\mcR$ a fibrant replacement
functor in $N$.

\begin{itemize}
\item Let $\mcK$ be a set of objects in $M$ with $F\mcQ\mcK$ its image in $N$. 
Then $F$ and $U$ induce a Quillen adjunction
\[
\adjunction{F}{\mcK \cell M}{F \mcQ \mcK \cell N}{U}
\]
between the $\mcK$-cellularisation of $M$ and the $F \mcQ \mcK$-cellularisation of $N$.

\item If $\mcK$ is a stable set of homotopically compact objects in $M$ such that for each $A$ in $\mcK$
the object $F\mcQ A$ is homotopically compact in $N$ and the derived unit $\mcQ A \to U \mcR F \mcQ A$ 
is a weak equivalence in $M$, then $F$ and $U$ induce a Quillen equivalence between the 
cellularisations:
\[
\mcK \cell M \simeq F \mcQ \mcK \cell N.
\]

\item If $\mcL$ is a stable set of homotopically compact objects in $N$ such that for each $B$ in $L$
the object $U \mcR B$ is homotopically compact in $M$ and the derived counit $F \mcQ U \mcR B \to \mcR B$ 
is a weak equivalence in $N$, then $F$ and $U$ induce a Quillen equivalence between the
cellularisations:
\[
U \mcR \mcL \cell M \simeq L \cell N.
\]
\end{itemize}
\end{thm}

\subsection{Alternatives to splitting}\label{sec:alternativetosplitting}
In the case of $SO(2)$, the rational Burnside ring is $\bQ$, 
so there are no idempotents to give a splitting. 
Instead, one must look for replacements for the idempotents 
or other methods of simplifying the category of rational $SO(2)$-spectra. 
One approach comes from inducing idempotents from the smaller subgroups. Suppose $H$ is a subgroup of $SO(2)$ such that $\A_\bQ(H)$ has an idempotent $e$. 
Then $SO(2)_+ \smashprod_H e \mathbb{S}$ is 
a retract of $SO(2)/H_+$
that does not come from an idempotent of $\A_\bQ(SO(2))$. 
The set of these spectra as $H$ and $e$ vary give a better 
behaved set of homotopically compact generators for rational $SO(2)$-spectra. 
We can think of this construction as applying an \emph{induced idempotent} 
to $SO(2)/H_+$. 
While they are not used directly in constructing the algebraic model 
for rational $SO(2)$-spectra, they are highly useful in understanding it.

Generalising the situation above, the rational Burnside ring of any torus $\T$ has no idempotents. Greenlees and Shipley \cite{tnqcore} provided a new method of obtaining an algebraic model in this case.
Suppose $\mcF$ is the family of all proper subgroups of $\T$, we define the universal space $E\mcF_+$ 
as a $\T$-CW-complex with the following universal property
\[
(E\mcF_+)^H \simeq \begin{cases}S^0 &\text{iff\ } H\in \mcF \\
* & \text{otherwise.} \end{cases}
\]
The universal space $E\mcF_+$ is part of a cofiber sequence called the \emph{isotropy separation sequence}
 \[
 \xymatrix{E\mcF_+ \ar[r]& S^0 \ar[r]& \widetilde{E}\mcF}
 \]
which can be turned into a homotopy pullback diagram in $\T$-spectra\footnote{We slightly abuse the notation and whenever we write a $\T$-space we actually mean its suspension spectrum.}.
In the case of $\T=SO(2)$, this is also called the \emph{Hasse square}:
 \[
 \xymatrix{ \mathbb{S} \ar[r]\ar[d]& \widetilde{E}\mcF \ar[d]\\
 DE\mcF_+ \ar[r]& DE\mcF_+ \wedge \widetilde{E}\mcF. }
 \]

The diagram with $\mathbb{S}$ removed is called the punctured cube and is denoted
by $\mathbb{S}^\lrcorner$. 
Using \cite{gsmodules}, we may construct a model category of 
modules over $\mathbb{S}^\lrcorner$ in rational $SO(2)$-spectra, which we call 
$\mathbb{S}^\lrcorner\text{-mod}$ (slightly abusing notation and not mentioning the ambient category). 
Any $SO(2)$-spectrum $X$ defines a module over the diagram by smashing with the 
ring spectra $\widetilde{E}\mcF$, $DE\mcF_+$ and $DE\mcF_+ \wedge \widetilde{E}\mcF$.
This functor has a right adjoint that is a type of pullback, giving an adjunction between 
$\mathbb{S}^\lrcorner\text{-mod}$ and rational $SO(2)$-spectra. 
The Cellularisation Principle, Theorem \ref{thm:cellprin},
can be used to construct a Quillen equivalence from this adjunction,
see either \cite[Sections 4-6]{tnqcore} or \cite[Section 3.2]{BGKSso2} for details.

In case of a torus of rank $r$, repeatedly using the isotropy separation sequence one can obtain a $r+1$-dimensional cube diagram. The terms of this cube are all  genuine-commutative equivariant ring $\T$-spectra
by Greenlees \cite{CouniversalCommutative_greenlees}.  We again use the notation $\mathbb{S}^\lrcorner$  for the punctured cube of these ring $\T$-spectra and obtain the following theorem.

\begin{thm}\label{thm:septorus}\cite[Proposition 4.1]{tnqcore}
There is a strong symmetric monoidal Quillen equivalence 
\[ 
L_{\mathbb{S}_\bQ}(\T \endash\mathrm{Sp}^\cO) \simeq_{QE} \  
\mathcal{K}\endash\mathrm{cell}\endash\mathbb{S}^\lrcorner\text{-mod}
 \]
 where  $\mathcal{K}$ is the image in $\mathbb{S}^\lrcorner\text{-mod}$ of the set of compact generators for $L_{\mathbb{S}_\bQ}(\T)\endash\mathrm{Sp}^\cO)$.
\end{thm}

When $G$ is a finite group, we let 
$\mcF$ be the family of all proper subgroups of $G$. 
The homotopy pullback diagram obtained by using the isotropy separation sequence gives exactly the idempotent splitting, since
\[
E\mcF_+\simeq \prod_{(H), \ H\in \mcF}e_H^G\mathbb{S} \simeq DE\mcF_+,
\] 
$\widetilde{E}\mcF \simeq e_G^G\mathbb{S}$ and  $DE\mcF_+\wedge \widetilde{E}\mcF \simeq *$.

However, the spectra $e_H^G\mathbb{S}$ are not genuine-commutative equivariant ring spectra
(they are only na\"ive-commutative).
Hence, it is easier to use the splitting approach for finite $G$.
See B{\"{o}}hme \cite{boehme19} for a complete explanation of the relation
between genuine-commutative equivariant ring spectra and localisation at idempotents.

An interesting case when there are some, but not enough idempotents, is the case of the dihedral
part of $O(2)$-spectra, see \cite{barneso2}.
In that case, there is no idempotent whose support is exactly $O(2)$. 
The abelian (resp. algebraic) model for the dihedral part of rational $O(2)$-spectra is given in terms 
of sheaves of $\bQ[W]$-modules (resp. differential $\bQ[W]$-modules) over the space $\tcD$, 
where the stalk over the point $O(2)$ has a trivial $W$-action. 
The stalk over $O(2)$ can be described in terms of a \emph{virtual idempotent} -- 
a colimit of idempotents, see \cite[Section 5]{barneso2}. 

A similar approach occurs for profinite groups in work of Barnes and Sugrue \cite{barnessugrue20}
and Sugrue \cite{sugruethesis}.

\section{Change of groups and localisations}\label{sec:change_localisation}

Once we split the category of rational $G$-spectra using idempotents, 
our main aim is to get rid of the remaining equivariance in each piece separately
by applying certain fixed points functors. 
Assume we are working with the category $L_{e\mathbb{S}_\bQ}(G\endash\mathrm{Sp}^\cO)$ and we want to take $H$ fixed points. First we must move to the category $N\endash\mathrm{Sp}^\cO$ where $N$ is the normaliser of $H$ in $G$, appropriately localised. We need $N$, since we want to have a residual Weyl group ($W=N_GH/H$) action. At the same time we need to localise $N\endash\mathrm{Sp}^\cO$ at some idempotent of the rational Burnside ring of $N$ corresponding to $e$, since we want to obtain a Quillen equivalence with $L_{e\mathbb{S}_\bQ}(G\endash\mathrm{Sp}^\cO)$.

In work of the second author \cite{KedziorekExceptional} and \cite{KedziorekSO(3)}, there was a precise analysis of two adjunctions: the induction--restriction and restriction--coinduction adjunctions in relation to localisations of categories of equivariant spectra at idempotents. Below we summarise how these results 
allow us to make the restriction--coinduction adjunction into Quillen equivalence
in suitable situations. Our examples are based on finite groups,  $O(2)$ and $SO(3)$.

\subsection{Restriction--coinduction adjunction and localisations}

Suppose we have an inclusion $i \colon N \hookrightarrow G$ of a subgroup $N$ in a group $G$. This gives a pair of adjoint functors at the level of orthogonal spectra (see for example \cite[Section V.2 ]{mm02}), namely induction, restriction and coinduction as below (the left adjoint is above the corresponding right adjoint). We note here, that for the induction functor to be a left Quillen functor we must take care over the universes involved.
\[
\xymatrix@C=6pc{
G\endash\mathrm{Sp}^\cO\
\ar@<-0ex>[r]|-(.5){\ i^*\ }
&
N\endash\mathrm{Sp}^\cO\
\ar@/^1pc/[l]^(.5){F_N(G_+,-)}
\ar@/_1pc/[l]_(.5){G_+\wedge_N -}
}
\]

We assume that $G$-spectra are indexed over a complete $G$-universe $\mathcal{U}$ 
and $N$-spectra are indexed over one of two universes. 
In the case where we want to use the restriction functor as a right adjoint, we use
the restriction of $\mathcal{U}$ to an $N$-universe.
If we consider restriction as a left Quillen functor 
we use a complete $N$-universe. 
With these conventions, the two pairs of adjoint functors are Quillen pairs 
with respect to stable model structures by \cite[Chapter V, Proposition 2.3 and 2.4]{mm02}. 
Given this, we slightly abuse the notation by not mentioning universes or 
the change of universe functors of \cite[Section V.2]{mm02}.

The restriction functor as a right adjoint is often used when we want to take 
(both categorical and geometric) $H$-fixed points of  $G$-spectra, 
where $H$ is not a normal subgroup of $G$. 
The procedure is to restrict to $N_G H$-spectra and then to take $H$-fixed points
to land in $W_G H$-spectra. 
This is usually done in one go, since the restriction functor and the $H$-fixed points 
functor are both right Quillen functors.

It is natural to ask when the pair of adjunctions above passes to the localised categories, in our case localised at $e_H^G\mathbb{S}_\bQ$ and $e_H^N\mathbb{S}_\bQ$ respectively.  The answer is related to $H$ being a good or bad subgroup in $G$. 
The induction--restriction adjunction does not always induce a Quillen adjunction on the localised categories,
unless $H$ is $N$-good in $G$. 
However, the restriction--coinduction adjunction induces a Quillen adjunction on these localised categories, for all exceptional subgroups $H$. Before we discuss this particular adjunction we state a general result.

\begin{lem}\label{locAdjAtObject} Suppose that $F: \cC \rightleftarrows \mcD: R$ is a Quillen adjunction of model categories where the left adjoint is strong (symmetric) monoidal. Suppose further that $E$ is a cofibrant object in $\cC$ and that both $L_E\cC$ and $L_{F(E)}\mcD$ exist. Then
\[
\xymatrix{
F \ :\ L_{E}\cC \  \ar@<+1ex>[r] & \  L_{F(E)}\mcD\ : U \ar@<+0.5ex>[l]
}
\]
is a strong (symmetric) monoidal Quillen adjunction. Furthermore,  
if the original adjunction was a Quillen equivalence, then
the induced adjunction on localised categories is as well.
\end{lem}
\begin{proof}Since the localisation did not change the cofibrations, the left adjoint $F$ still preserves them. To show that it also preserves acyclic cofibrations, take an acyclic cofibration ${f \colon  X\lra Y}$ in $L_E\cC$. By definition, 
$f\wedge \id_E$ is an acyclic cofibration in $\cC$. 
Since $F$ was a left Quillen functor before localisation, $F(f\wedge \id_E)$ is an acyclic cofibration in $\mcD$. 
As $F$ was strong  monoidal, we have $F(f\wedge \id_E)\cong F(f)\wedge \id_{F(E)}$, 
so $F(f)$ is an acyclic cofibration in $L_{F(E)}\mcD$ which finishes the proof of the first part.

To prove the second part of the statement we use Part (2) from \cite[Corollary 1.3.16]{hov99}.
Since $F$ is strong monoidal, and the original adjunction was a Quillen equivalence, 
$F$ reflects $F(E)$-equivalences between cofibrant objects. It remains to check that the derived counit is an $F(E)$-equivalence. An $F(E)$-fibrant object is fibrant in $\mcD$ and the cofibrant replacement functor remains unchanged by localisation. Thus the claim follows from the fact that $(F,U)$ was a Quillen equivalence before localisations.
\end{proof}

We will use this result in several cases. We start with the restriction--coinduction adjunction.

\begin{cor}\label{cor:coindQuillenadj} Let $i \colon N \lra G$ denote the inclusion of a subgroup and let $E$ be a cofibrant object in $G\endash\mathrm{Sp}^\cO$. Then
\[
\xymatrix{
i^\ast \ :\ L_{E}(G\endash\mathrm{Sp}^\cO) \  \ar@<+1ex>[r] & \  L_{i^\ast(E)}(N\endash\mathrm{Sp}^\cO)\ : F_N(G_+,-) \ar@<+0.5ex>[l]
}
\]
is a strong symmetric monoidal Quillen pair.
\end{cor}

Notice that if $E=e\mathbb{S}_\bQ$ for some idempotent $e \in \A_\bQ(G)$ then we get the following

\begin{cor}\label{QuillenAdjForIandH} Suppose $G$ is any compact Lie group, $i \colon N \lra G$ is an inclusion of a subgroup and $e$ is an idempotent in $\A_\bQ(G)$. Then the adjunction
\[
\xymatrix{
i^\ast\ :\ L_{e\mathbb{S}_\bQ}(G\endash\mathrm{Sp}^\cO)\ \ar@<+1ex>[r] & L_{i^*(e)\mathbb{S}_\bQ}(N\endash\mathrm{Sp}^\cO)\ :\ F_N(G_+,-) \ar@<+0.5ex>[l]
}
\]
is a Quillen pair.
\end{cor}

\subsection{Exceptional part of rational \texorpdfstring{$G$}{G}-spectra}\label{subs:exceptional_res_coind}

We will repeatedly use the above result, mainly in situations where after further localisation of the right hand side we will get a Quillen equivalence.

\begin{cor}\label{localisedQAdjunctions}\cite[Corollary 4.7]{KedziorekExceptional}
 Suppose $G$ is a compact Lie group and $H$ is an exceptional subgroup of $G$.  Then 
\[
\xymatrix{
i^\ast\ :\ L_{e_H^G\mathbb{S}_\bQ}(G\endash\mathrm{Sp}^\cO)\ \ar@<+1ex>[r] & L_{e_H^N\mathbb{S}_\bQ}(N\endash\mathrm{Sp}^\cO)\ :\ F_N(G_+,-) \ar@<+0.5ex>[l]
}
\]
is a Quillen pair.

\end{cor}
\begin{proof} For $N$-good $H$, the result follows from the fact that the idempotent on the right hand side satisfies $e_H^N=i^*(e_H^G)$. For $N$-bad $H$, it is true since the left hand side is a further localisation of $L_{i^*(e_H^G)\mathbb{S}_\bQ}(N-\mathrm{Sp}^\cO)$ at the idempotent $e_H^N$:
\[
\xymatrix@C=3pc{
L_{e_H^G\mathbb{S}_\bQ}(G\endash\mathrm{Sp}^\cO)\
\ar@<+1ex>[r]^{i^\ast}
&
\ L_{i^*(e_H^G) \mathbb{S}_{\bQ}}(N\endash\mathrm{Sp}^\cO)\
\ar@<+0.5ex>[l]^{F_{N}(G_+,-)}
\ar@<+1ex>[r]^(.53){\id}
&
\ L_{e_H^N\mathbb{S}_\bQ}(N-\mathrm{Sp}^\cO)
\ar@<+0.5ex>[l]^(.47){\id}
}
\]
Note that since $H$ is bad, $e_H^N \neq i^*(e_H^G)$ and $e_H^N i^*(e_H^G)=e_H^N$.
\end{proof}

\begin{thm}\label{badLeftAdjoint}\cite[Theorem 4.8]{KedziorekExceptional} Suppose $H$ is an exceptional subgroup of $G$. Then the adjunction
\[
\xymatrix@C=3pc{
i^\ast_N\ :\ L_{e_H^G\mathbb{S}_{\bQ}}(G\endash\mathrm{Sp}^\cO)\
\ar@<+1ex>[r]
&
\ L_{e_H^N \mathbb{S}_{\bQ}}(N\endash\mathrm{Sp}^\cO)\ :\ F_N(G_+,-)
\ar@<+0.5ex>[l]
}
\]
is a strong symmetric monoidal Quillen equivalence. 
\end{thm}

Part of the difficulty in providing an algebraic model for a \emph{piece} of homotopy category of rational $G$-spectra governed by an idempotent $e$ comes from balancing two things. On the one hand, one wants to simplify the ambient category as much as one can. On the other one must preserve all the relevant homotopical information.
This balancing act requires deep understanding of the homotopy category of $L_{e\mathbb{S}_{\bQ}} G\endash\mathrm{Sp}$.
In the case of an exceptional subgroup $H$ of $G$, this is achieved by passing to $L_{e_H^{N}\mathbb{S}_{\bQ}}N\endash\mathrm{Sp}$ using restriction as a left Quillen functor, as we described above.

There was a reason why we considered restriction to be a left Quillen functor and it is related to the good and bad exceptional subgroups in $G$. 

\begin{prop}\cite[Proposition 4.5]{KedziorekExceptional}\label{notQErest_ind_excep} Suppose $H$ is an exceptional subgroup of $G$ that is $N$-bad in $G$. Then
\[
\xymatrix{
i^\ast\ :\ L_{e_H^G\mathbb{S}_\bQ}(G\endash\mathrm{Sp}^\cO)\ \ar@<-1ex>[r] & L_{e_H^N\mathbb{S}_\bQ}(N\endash\mathrm{Sp}^\cO)\ :\ G_+\wedge_N- \ar@<-0.5ex>[l]
}
\]
is not a Quillen adjunction. If $H$ is an $N$-good subgroup in $G$, then 
the above adjunction is a Quillen pair.
\end{prop}

\subsection{Toral part of rational \texorpdfstring{$SO(3)$}{SO(3)}-spectra}

The functor $i^*$ is not always a right Quillen functor 
when considered between categories localised at the toral idempotents. 
One can argue that this is because the toral idempotents do not always correspond with each other. 
One example is when $G=SO(3)$, $\T=SO(2)$ and $N=O(2)$. In that case the proof is based on the fact that $D_2$ is conjugate to $C_2$ in $SO(3)$ and thus $i^*(e_{\cT})\neq e_{\tcT}$.

\begin{prop}\label{notQErest_ind}\cite[Proposition 2.7]{KedziorekSO(3)} Suppose $e_{\cT}$ is the toral idempotent of $SO(3)$ and $e_{\tcT}$ is the toral idempotent of $O(2)$.That is, $e_{\cT}$ is the  idempotent in $\A_\bQ(SO(3))$ corresponding to the characteristic function of the toral part $\cT$ (i.e. all subconjugates of the maximal torus of $SO(3)$) and $e_{\tcT}$ is the idempotent in $\A_\bQ(O(2))$ corresponding to the characteristic function of the toral part $\tcT$, i.e. all  subconjugates of the maximal torus of $O(2)$ (see Sections  \ref{sec:subgroupsO2} and \ref{sec:subgroupsSO3}). Then
\[
\xymatrix{
i^\ast\ :\ L_{e_{\cT}\mathbb{S}_\bQ}(SO(3)\endash\mathrm{Sp})\ \ar@<-1ex>[r] & L_{e_{\tcT}\mathbb{S}_\bQ}(O(2)\endash\mathrm{Sp})\ :\ SO(3)_+\wedge_{O(2)}- \ar@<-0.5ex>[l]
}
\]
is not a Quillen adjunction.
\end{prop}

The restriction--coinduction adjunction is often better behaved with respect to localisation at idempotents. 

\begin{prop}\label{cyclicAdj}Let $i \colon O(2) \lra SO(3)$ be the inclusion. Then the following adjunction
\[
\xymatrix{
i^\ast \ :\ L_{e_{\cT}\mathbb{S}_\bQ}(SO(3)\endash\mathrm{Sp})\  \ar@<+1ex>[r] & \ L_{e_{\tcT}\mathbb{S}_\bQ}(O(2)\endash\mathrm{Sp}) \ :\ F_{O(2)}(SO(3)_+,-) \ar@<+0.5ex>[l]
}
\]
is a strong symmetric monoidal Quillen adjunction. 
\end{prop}

The proof follows the same argument as Corollary \ref{localisedQAdjunctions} above, in the sense that 
the adjunction is a composite of the restriction--coinduction adjunction localised at an idempotent $e_{\cT}$ 
(and its restriction $i^\ast(e_{\cT})$) followed by a further localisation of $O(2)$-spectra (which excludes the subgroup $D_2$).

This adjunction of restriction and coinduction is not quite a Quillen equivalence. However cellularising the right hand side at the derived images of the homotopically compact generators $\mathcal{K}$ for rational toral $SO(3)$-spectra and using the Cellularisation Principle (see Theorem \ref{thm:cellprin}) gives a Quillen equivalence.

\begin{thm}\label{thm:so3so2adjuncation}\cite[Theorem 3.28]{KedziorekSO(3)}
The following adjunction
\[
\xymatrix{
i^\ast \ :\ L_{e_{\cT} \mathbb{S}_\bQ}(SO(3)\endash\mathrm{Sp})\  \ar@<+1ex>[r] & \ i^\ast(\mathcal{K})\endash\mathrm{cell}\endash L_{e_{\tcT} \mathbb{S}_\bQ}(O(2)\endash\mathrm{Sp}) \ : F_{O(2)}(SO(3)_+,-) \ar@<+0.5ex>[l]
}
\]
is a Quillen equivalence, where $\mathcal{K}$ denotes the set of homotopically compact generators for the model category $L_{e_{\cT} \mathbb{S}_\bQ}(SO(3)\endash\mathrm{Sp})$.
\end{thm}

\subsection{Dihedral part of rational \texorpdfstring{$SO(3)$}{SO(3)}-spectra}

In other cases of idempotents it is not always clear to which category one should restrict. 
For the dihedral idempotent in rational $SO(3)$-spectra, 
restricting to certain part of the rational dihedral $O(2)$-spectra is the correct choice, 
but in general there is no good recipe for obtaining an algebraic model.

In the dihedral part of $SO(3)$ we can use restriction as a right or left Quillen functor, we chose the following one, which also follows from Lemma \ref{locAdjAtObject}.

\begin{cor}
Let $\mcD$ denote the dihedral part of $SO(3)$ and $e_{\mcD}$ the corresponding idempotent. 
Let $i \colon O(2) \lra SO(3)$ be the  inclusion. Then
\[
\xymatrix{
i^\ast\ :\ L_{e_{\mcD}\mathbb{S}_\bQ}(SO(3)\endash\mathrm{Sp})\ \ar@<+1ex>[r] & L_{i^*(e_{\mcD})\mathbb{S}_\bQ}(O(2)\endash\mathrm{Sp})\ :\ F_{O(2)}(SO(3)_+,-) \ar@<+0.5ex>[l]
}
\]
is a Quillen adjunction.
\end{cor}

\begin{rmk}\label{rmk:different_localisations} The idempotent on the right hand side $i^*(e_{\mcD})$ corresponds to the dihedral part of $O(2)$ \emph{excluding} all subgroups $D_2$ and $D_4$. Thus, $i^*(e_{\mcD}) = i^*(e_{\mcD}) e_{\tcD}$.
\end{rmk}

\subsection{Inflation and fixed point adjunction}\label{subs:infl_fixedpts}

Suppose $H$ is a normal subgroup of $N$ and consider the natural projection $\varepsilon \colon N\lra N/H=W$. 
Then there is a pair of adjoint functors
\[
\xymatrix{
\varepsilon^\ast \ :\ W\endash\mathrm{Sp}^\cO \  \ar@<+1ex>[r] & \  N\endash\mathrm{Sp}^\cO\ : (-)^H \ar@<+0.5ex>[l]
}
\]
where the right adjoint is the $H$ fixed points functor and the left adjoint is called inflation. 
For details see \cite[Section V.3]{mm02}.

We would like to understand the interaction between the localisation at idempotents and the above adjunction. Notice that since inflation is strong symmetric monoidal, the result below follows from Lemma \ref{locAdjAtObject}.

\begin{cor}Let $\varepsilon \colon N \lra W$ denote the projection of groups, where $H$ is normal in $N$ and $W=N/H$. Let $E$ be a cofibrant object in $W\endash\mathrm{Sp}^\cO$. Then
\[
\xymatrix{
\varepsilon^\ast \ :\ L_{E}(W\endash\mathrm{Sp}^\cO) \  \ar@<+1ex>[r] & \  L_{\varepsilon^\ast(E)}(N\endash\mathrm{Sp}^\cO)\ : (-)^H \ar@<+0.5ex>[l]
}
\]
is a strong symmetric monoidal Quillen pair.
\end{cor}

\begin{lem}\label{lem:exceptionalfixed}\cite[Theorem 5.2]{KedziorekExceptional}
For $H$ an exceptional subgroup of $N$, the adjunction
\[
\xymatrix{
\varepsilon^\ast \ :\ L_{e_1^W \mathbb{S}_\bQ}(W\endash\mathrm{Sp}^\cO) \  \ar@<+1ex>[r] & \  L_{e_H^N \mathbb{S}_\bQ}(N\endash\mathrm{Sp}^\cO)\ : (-)^H \ar@<+0.5ex>[l]
}
\]
is a Quillen equivalence. Here $e_1^W$ denotes an idempotent for the trivial subgroup $\{1\} \leqslant W$.
\end{lem}

In case of a torus $\T$,  
we define $(\mathbb{S}^\lrcorner)^\T$ to be the diagram of commutative ring spectra obtained by taking 
objectwise $\T$-fixed points of $\mathbb{S}^\lrcorner$ (from Section \ref{sec:alternativetosplitting}).
We illustrate this in the case $\T=SO(2)$.
  \[
  \mathbb{S}^\lrcorner =  
  \left(
  \begin{gathered}
  \xymatrix{\ & \widetilde{E}\mcF \ar[d]\\
  DE\mcF_+ \ar[r]& DE\mcF_+ \wedge \widetilde{E}\mcF }
  \end{gathered}
  \right)
  \quad \quad 
  (\mathbb{S}^\lrcorner)^\T = 
  \left( 
  \begin{gathered} \xymatrix{\ & \widetilde{E}\mcF^\T \ar[d]\\
  DE\mcF_+^\T \ar[r]& (DE\mcF_+ \wedge \widetilde{E}\mcF)^\T } 
  \end{gathered}
  \right)
  \]
  
The inflation--fixed point adjunction lifts to the level of module categories over 
the diagrams of rings $\mathbb{S}^\lrcorner$ and  $(\mathbb{S}^\lrcorner)^\T$ by \cite{gsfixed}.
This adjunction is a Quillen equivalence and by the 
Cellularisation Principle, Theorem \ref{thm:cellprin}, it induces a 
Quillen equivalence on the cellularised categories as follows. 
We refer the reader to \cite[Section 7]{tnqcore} for more details.
 
\begin{thm}\label{thm:torusfixedpoints} Let $\T$ be a torus.
The fixed point functor induces strong symmetric monoidal Quillen equivalences  
\begin{align*}
\mathbb{S}^\lrcorner\text{-mod}
& \simeq_{QE}
(\mathbb{S}^\lrcorner)^\T\text{-mod} \\
\mathcal{K}\endash\mathrm{cell}\endash\mathbb{S}^\lrcorner\text{-mod} 
& \simeq_{QE}
\mathcal{K}^\T\endash\mathrm{cell}\endash(\mathbb{S}^\lrcorner)^\T\text{-mod}
\end{align*}
 where $\mathcal{K}$ is the image in $\mathbb{S}^\lrcorner\text{-mod}$ of the set of compact generators for $L_{\mathbb{S}_\bQ}(\T\endash\mathrm{Sp}^\cO)$ and
$\mathcal{K}^\T$ its image in $(\mathbb{S}^\lrcorner)^\T\text{-mod}$. 
\end{thm}

The advantage of this last theorem is that it gives
a model for rational $\T$-spectra in terms of non-equivariant spectra.

The base idea for the toral part of rational $N$-spectra (where $\T$ is normal in $N$)
is to use the same steps, but in a context where after taking $\T$-fixed points we 
land in a category of spectra with an action of $W=N/\T$. 
This requires some very detailed constructions
to make precise, which we leave to \cite{toralBGK}.

\section{An algebraic model for rational \texorpdfstring{$G$}{G}-spectra - overview of some cases}

In this section we provide a summary of the necessary steps  to obtain an algebraic model for a (part of) rational $G$-spectra in two cases. The first case is when $G$ is a finite group and we follow 
the steps presented in the algebraic case in Part \ref{part:1}. 
The second case is when we are interested in the toral part of rational $G$-spectra, for any compact Lie group $G$.
We discuss briefly  the series of simplifications required for the classification result in this case.

\subsection{An algebraic model for rational \texorpdfstring{$G$}{G}-spectra for finite \texorpdfstring{$G$}{G}}

Building on the results of Sections \ref{subs:exceptional_res_coind} and \ref{subs:infl_fixedpts} we can sketch the passage to the algebraic model for rational $G$-spectra when $G$ is a finite group.

Theorem \ref{thm:splittingofexep} allows us to split the category of rational $G$-spectra into a finite product
\[
\prod_{(H) \leqslant G} L_{e_H^G \mathbb{S}_\bQ}(G\endash\mathrm{Sp}^\cO).
\]

The next step uses the restriction--coinduction Quillen equivalence 
\[L_{e_H^G \mathbb{S}_\bQ}(G\endash\mathrm{Sp}^\cO) \simeq_{QE} L_{e_H^N \mathbb{S}_\bQ}(N\endash\mathrm{Sp}^\cO) 
\]
for each factor of the product seperately. We then follow with the inflation-fixed point Quillen equivalence
\[L_{e_H^N \mathbb{S}_\bQ}(N\endash\mathrm{Sp}^\cO) \simeq_{QE} L_{e_1^W \mathbb{S}_\bQ}(W\endash\mathrm{Sp}^\cO) 
\]
of the previous section.

The model category $L_{e_1^W \mathbb{S}_\bQ}(W\endash\mathrm{Sp}^\cO)$ obtained after taking $H$-fixed points of $L_{e_H^N \mathbb{S}_\bQ}(N\endash\mathrm{Sp}^\cO)$ can be described in a much easier way. It is Quillen equivalent to the model category $\mathrm{Sp}^\cO[W]$ of orthogonal spectra with the $W$ action, where the model structure is created from the one on $\mathrm{Sp}^\cO$ by the forgetful functor $U \colon \mathrm{Sp}^\cO[W] \lra \mathrm{Sp}^\cO$.
This allows us to remove the equivariance from \emph{inside} of the complicated category $W\endash\mathrm{Sp}^\cO$ (where it appeared in the indexing spaces for the spectrum) to the \emph{outside} of much simpler $\mathrm{Sp}^\cO[W]$.

Shipley \cite{shiHZ} gives a (zig-zag of weak)
symmetric monoidal Quillen equivalences between rational spectra and 
chain complexes of $\bQ$-modules (with the projective model structure). 
This is often referred to in the literature as a \emph{algebraicisation}.
This result readily extends to a Quillen equivalence between 
rational spectra with a finite group action and rational chain complexes with a 
finite group action. 
Hence, we obtain an algebraic model for $L_{e_H^N \mathbb{S}_\bQ}(N\endash\mathrm{Sp}^\cO)$ in terms of chain complexes of $\bQ[W_G H]$-modules. 

Combining all the steps mentioned in this section we obtain the following result.

\begin{thm}\label{thm:finitealgmodel}
For $G$ a finite group, there is a zig-zag of symmetric monoidal Quillen equivalences between 
$L_{e_H^G \mathbb{S}_\bQ}(G\endash\mathrm{Sp}^\cO)$
and $\Ch(\bQ[W_G H])$.

The algebraic model for rational $G$-spectra is therefore
\[
\prod_{(K) \leqslant G} \Ch(\bQ[W_G K]).
\]
Moreover, if $X$ is a rational $G$-spectrum with corresponding 
object $(A_K)_{(K) \leqslant G}$ in the algebraic model, then 
\[
\pi_* ( \big(i^* (e_K^G X)\big)^K ) 
\cong 
\pi_*(\Phi^K X)
\cong 
H_*(A_K).
\]
Here $i^*$ and $(-)^K$ denote derived functors of restriction and fixed points discussed in Sections \ref{subs:exceptional_res_coind} and \ref{subs:infl_fixedpts}, respectively and $H_*$ denotes homology.
\end{thm}

\subsection{Morita equivalences}
A different approach to obtaining an algebraic model for rational $G$-spectra for a finite group $G$ is presented in \cite{barnesfinite} and uses Morita equivalences developed in the spectral setting by Schwede and Shipley \cite{ss03stabmodcat}.

The idea is to present $L_{e_H^G \mathbb{S}_\bQ}(G\endash\mathrm{Sp}^\cO)$
as a category of modules over the endomorphism ring spectrum of the compact generator $e_H^GG/H_+$. 

Let $\underline{\Hom}(-,=)$ denote the enrichment of $G$-spectra in non-equivariant spectra.  
Then 
\[
E_H=\underline{\Hom}(e_H^GG/H_+,e_H^GG/H_+)
\] 
(with fibrant replacements omitted from the notation)
is a ring spectrum under composition.
Furthermore, the model category of modules over $E_H$ (in non-equivariant spectra)
is Quillen equivalent to $L_{e_H^G \mathbb{S}_\bQ}(G\endash\mathrm{Sp}^\cO)$.
One can then use algebraicisation (the results of Shipley \cite{shiHZ}) to obtain an algebraic model 
for this part of rational $G$-spectra. 

 However, $E_H$ it is not (in general) a \emph{commutative} ring spectrum in 
orthogonal spectra. 
The problem is fundamental and can be seen by looking at homotopy groups. 
The homotopy groups of $E_H$ are non-trivial only in degree $0$, where
they take value
\[
\pi_0(\underline{\Hom}(e_H^GG/H_+,e_H^GG/H_+))= \bQ [W_GH].
\]
While this has a cocommutative Hopf algebra structure, 
it does not have a commutative ring structure in general. 

This makes it much harder to obtain a comparison that takes into account the monoidal structures. 
In particular, we would need to check that the algebraicisation of the ring spectrum $E_H$
also has a cocommutative Hopf algebra structure. 
As we only have control over the homology of the algebraicised object, 
we would also need a formality argument that preserves the 
cocommutative Hopf algebra structure. 

\subsection{An algebraic model for the toral part of rational \texorpdfstring{$G$}{G}-spectra}

We give a brief overview of the remaining steps needed to classify 
rational toral $G$-spectra. Details are left to the references. 
While reading the summary, the reader may like to keep in mind the case $\T=SO(2)$, $G=SO(3)$ and $N=O(2)$.
These are the easiest cases of interest and have been discussed in previous sections.

Greenlees and Shipley \cite{tnqcore} give an algebraic model for 
rational $\T$-spectra, where $\T$ is a torus.
See also \cite{gre99} for a full explanation of the algebraic model
and \cite{BGKSso2} for the classification of rational $SO(2)$-spectra. 
The first two steps of the classification are to apply 
Theorems \ref{thm:septorus} and \ref{thm:torusfixedpoints}. 
The next step is to algebraicise using work of Shipley \cite{shiHZ}.
This gives an algebraic model for rational $\T$-spectra in terms of 
(a cellularisation of) a category of modules 
over a diagram of commutative differential graded algebras. 
Formality of these commutative dgas allows us to simplify the 
rings in the diagram. An additional simplification of the algebra 
removes the cellularisation and 
gives the algebraic model for rational $\T$-spectra. 

Work of the authors and Greenlees gives an algebraic model
for the toral part of rational $G$-spectra for any compact Lie group $G$, see \cite{toralBGK}. 
Given $G$, we let $\T$ be a maximal torus and $N$ its normaliser in $G$.
We can lift the classification for rational $\T$-spectra to 
a classification of rational toral $N$-spectra.
We then use the following result to reduce 
problem of classifying rational  toral $G$-spectra to 
understanding a cellularisation of rational toral $N$-spectra. 

\begin{thm}\cite[Theorem 2.2]{toralBGK}\label{thm:generalToral}
The following adjunction
\[
\xymatrix{
i^\ast \ :\ L_{e_{\T}^G \mathbb{S}_\bQ}(G\endash\mathrm{Sp})\  \ar@<+1ex>[r] & \ i^\ast(\mathcal{L})\endash\mathrm{cell}\endash L_{e_{\T}^N \mathbb{S}_\bQ}(N\endash\mathrm{Sp}) \ : F_{N}(G_+,-) \ar@<+0.5ex>[l]
}
\]
is a Quillen equivalence, where the idempotent on both sides corresponds to the families of all subgroups of maximal torus $\T\leq N \leq G$ and $\mathcal{L}$ denotes the set of homotopically compact generators for $L_{e_{\T}^G \mathbb{S}_\bQ}(G\endash\mathrm{Sp})$.
\end{thm}

By the Cellularisation Principle, Theorem \ref{thm:cellprin}, 
we can cellularise each term of the classification of rational toral $N$-spectra at the 
derived images of the cells $\mathcal{L}$.
This gives a classification of rational toral $G$-spectra in terms of 
a cellularisation of the algebraic model for rational toral $N$-spectra. 
The final simplification is to remove this cellularisation,
which is based on another formality argument.

\bibliographystyle{alpha}
\bibliography{ourbib}

\end{document}